\documentclass[11pt,a4paper,reqno]{amsart}
\usepackage[active]{srcltx}
\usepackage{a4wide}
\usepackage{amssymb, amsmath, mathtools}
\usepackage{graphicx}
\usepackage{float}
\usepackage[active]{srcltx}
\usepackage[ansinew]{inputenc}
\usepackage{hyperref}
\usepackage{color}
\usepackage{epsfig}
\theoremstyle{plain}
\newtheorem{theorem}{Theorem}[section]
\newtheorem{maintheorem}{Theorem}
\newtheorem{lemma}[theorem]{Lemma}

\newtheorem{corollary}[theorem]{Corollary}

\theoremstyle{remark}

\numberwithin{equation}{section}

\newcommand{\NN}{{\mathbb{N}}}

\newcommand{\ZZ}{{\mathbb{Z}}}
\newcommand{\RR}{{\mathbb{R}}}

\newcommand{\EU}{{\rm\bf S}}
\newcommand{\vv}{{\rm\bf v}}

\newcommand{\C}{{\mathcal C}}

\newcommand{\loc}{\textnormal{loc}}
\newcommand{\Inn}{\textnormal{In}}
\newcommand{\Out}{\textnormal{Out}}

\begin{document}

\title[Moduli of stability]{Moduli of stability for heteroclinic cycles \\ of periodic solutions}

\author[M. Carvalho]{Maria Carvalho}
\address{Maria Carvalho\\ Centro de Matem\'{a}tica da Universidade do Porto\\ Rua do
Campo Alegre 687\\ 4169-007 Porto\\ Portugal}
\email{mpcarval@fc.up.pt}

\author[A. Lohse]{Alexander Lohse}
\address{Alexander Lohse\\ Fachbereich Mathematik\\ Universit\"at Hamburg\\ Bundesstra{\ss}e 55\\ 20146 Hamburg\\ Germany}
\email{alexander.lohse@math.uni-hamburg.de}

\author[A. Rodrigues]{Alexandre A. P. Rodrigues}
\address{Alexandre Rodrigues\\ Centro de Matem\'{a}tica da Universidade do Porto\\ Rua do
Campo Alegre 687\\ 4169-007 Porto\\ Portugal}
\email{alexandre.rodrigues@fc.up.pt}

\date{\today}
\thanks{MC and AR were partially supported by CMUP (UID/MAT/00144/2019), which is funded by FCT with national (MCTES) and European structural funds through the programs FEDER, under
the partnership agreement PT2020. AR also acknowledges financial support from Program INVESTIGADOR FCT (IF/00107/2015).   Part of this work has been written during AR's stay in Nizhny Novgorod University, supported by the grant RNF 14-41-00044.
Visits to Porto by AL were funded through project 57338573  PPP  Portugal  2017  of  the  German  Academic  Exchange  Service
(DAAD),  sponsored  by  the  Federal  Ministry  of  Education  and  Research (BMBF)}

\maketitle

\setcounter{tocdepth}{2}

\begin{abstract}
We consider $C^2$ vector fields in the three dimensional sphere with an attracting heteroclinic cycle between two periodic hyperbolic solutions with real Floquet multipliers. The proper basin of this attracting set exhibits historic behavior and from the asymptotic properties of its orbits we obtain a complete set of invariants under topological conjugacy in a neighborhood of the cycle. As expected, this set contains the periods of the orbits involved in the cycle, a combination of their angular speeds, the rates of expansion and contraction in linearizing neighborhoods of them, besides information regarding the transition maps and the transition times between these neighborhoods. We conclude with an application of this result to a class of cycles obtained by the lifting of an example of R. Bowen.
\end{abstract}

\section{Introduction}

In the study of dynamical systems it has long been of interest to identify systems that display similar behavior in the sense that their phase diagrams look qualitatively the same. For continuous systems $\dot{x}=f(x)$ given by some vector field $f$, this amounts to deciding under what conditions the flows generated by two different vector fields are topologically equivalent or even conjugate. In particular, it is desirable to find quantities of the system that are invariant under topological conjugacy and, moreover, fully characterize conjugacy classes of systems through a (minimal) number of these quantities. Such a collection is then called a \emph{complete} set of invariants.

In the context of heteroclinic dynamics, significant contributions to this type of question have been made by several authors. We briefly review the invariants under conjugacy that have been found for: (a) heteroclinic connections between equilibria; (b) attracting heteroclinic cycles between equilibria; and (c) heteroclinic connections associated to one periodic solution.
As far as we know, the description of complete sets of invariants for attracting heteroclinic cycles associated to periodic solutions has not yet been done.

For heteroclinic connections, Dufraine \cite{Dufraine}, building on the work of Palis \cite{Palis}, considers one-dimensional heteroclinic connections between two hyperbolic equilibria on a three-dimensional manifold, each with one real and one pair of complex conjugated eigenvalues. He finds a set of invariants involving two quantities: the ratio of the real parts of the complex eigenvalues, and an expression combining this ratio with their imaginary parts. Bonatti and Dufraine \cite{BonDuf} go on to extend this result to obtain a complete characterization of such a heteroclinic connection up to topological equivalence. Higher dimensional heteroclinic connections between equilibria are analyzed in a similar way by Susín and Simó \cite{SuSi}.

Takens \cite{Takens94} provides analogous investigations for an attracting heteroclinic cycle with two one-dimensional connections between hyperbolic equilibria, this time with only real eigenvalues. Under the assumption that the transitions between suitable cross sections to the cycle is instantaneous and the global maps are linear, he finds a complete set of three invariants that are intuitively compatible with the ones mentioned above: two ratios of eigenvalues as found by Palis \cite{Palis}, plus an expression relating these to properties of the global transition map. Completeness is proved by constructing a conjugacy based on asymptotic properties of Birkhoff time averages -- a technique we also use in this paper.

Carvalho and Rodrigues \cite{CR2017} consider a Bykov attractor -- a heteroclinic cycle between two hyperbolic equilibria on a three-dimensional sphere with a one-dimensional connection as in \cite{Dufraine} and a two-dimensional connection as in \cite{SuSi} between them. Extending the argument of \cite{Takens94}, they find a complete set of four invariants for this situation, namely a combination of the angular speeds of the equilibria, the rates of expansion and contraction in linearizing neighborhoods of them, besides information regarding the transition maps between these neighborhoods. See their paper also for a more detailed overview of the previous results that we mentioned here only briefly.

Beloqui \cite{Beloqui} considers a one-dimensional connection between a saddle-focus equilibrium and a periodic solution and derives an invariant under conjugacy. More precisely, Beloqui studies a heteroclinic connection associated to a saddle-focus $p$ (with eigenvalues $-C_p \pm i\omega$ and $E_p$) and a periodic solution $\mathcal{P}$ (with minimal period $\wp$ and real Floquet exponents $C_\mathcal{P}$ and $E_\mathcal{P}$ such that $|C_\mathcal{P}| < 1$ and $|E_\mathcal{P}| > 1$) and shows that $\frac{C_p}{\omega \,E_\mathcal{P}} $ is a topological invariant. By a similar argument but under additional assumptions, Rodrigues \cite{Rodrigues2015} obtains a new invariant, given by
$$\frac{1}{E_\mathcal{P}+C_p} \left(\omega \,E_\mathcal{P} + \frac{2\pi}{\wp}\,C_p \right).$$

Our contribution lies in combining and extending techniques used in the previous works to address the question of complete sets of topological invariants for attracting heteroclinic cycles with two-dimensional connections between two hyperbolic periodic solutions with real Floquet multipliers (called ``PtoP" cycle). From the asymptotic properties of the orbits, the transition maps and the transition times between linearizing neighborhoods of the periodic solutions, we obtain a complete set of invariants under topological conjugacy in the basin of attraction of the cycle. Unsurprisingly, the eight invariants we find include the two minimal periods of the periodic solutions; the other six are closely related to those found in earlier works. They reduce to those found in \cite{CR2017} under the assumptions therein on the global transitions (which we are able to loosen here).

While our results are primarily of interest in terms of further understanding and classifying heteroclinic behavior from an abstract point of view, heteroclinic cycles between periodic solutions appear in several models of real-life systems: for instance, Zhang, Krauskopf and Kirk \cite{ZKK} consider a four-dimensional model for intracellular calcium dynamics where a codimension one ``PtoP" cycle between two periodic solutions appears. Their setup differs from our situation, though, by one of the connections being one-dimensional.

This paper is structured as follows. In Sections~\ref{se:hypotheses} and \ref{sse:definitions} we introduce the setting and establish some notation. Section~\ref{se:mainresults} states our main result, giving a complete list of invariants under topological conjugacy for a ``PtoP" heteroclinic cycle. In Sections~\ref{se:Local} and \ref{se:hitting-times} we analyze the local and global dynamics near the cycle as well as the hitting times of the trajectories attracted to it. The proof of our main theorem is spread over Sections~\ref{se:invariants} and \ref{se:completeness}, where we derive the invariants and prove that they indeed form a complete set. We conclude with an example in Section~\ref{Example}, obtained by the lift of a well-known system studied in \cite{Takens94} and attributed to Bowen.

\section{The setting}\label{se:hypotheses}

We consider $C^2$ vector fields $f: \EU^3 \rightarrow T\EU^3$ on the unit sphere $\EU^3$ and the corresponding differential equations $\dot{x}=f(x)$ subject to initial conditions $x(0)=x_0\in \EU^3$. We will assume that $f$ has the following properties:

\begin{enumerate}
\item[(\textbf{P1})] There are two hyperbolic periodic solutions $\C_1$ and $\C_2$ of saddle-type, with minimal periods $\wp_1$ and $\wp_2$, within which the flow has constant angular speed $\omega_1>0$ and $\omega_2>0$, respectively. The Floquet multipliers of $\C_1$ and $\C_2$ are real and given by
\begin{eqnarray*}
e^{E_1} > 1 \quad &\text{and}& \quad e^{-C_1} < 1 \quad \quad \text{for} \quad \C_1 \\
e^{E_2} >1 \quad &\text{and}& \quad e^{-C_2} < 1  \quad \quad \text{for} \quad \C_2
\end{eqnarray*}
where $C_1 > E_1$ and $C_2 > E_2$.
\medskip

\item[(\textbf{P2})] The stable manifolds $W^s_{\loc}(\C_{1}), \,W^s_{\loc}(\C_{2})$ and the unstable manifolds $W^u_{\loc}(\C_1), \,W^u_{\loc}(\C_2)$ are smooth surfaces homeomorphic to a cylinder.
\medskip

\item[(\textbf{P3})] For every $j \in \{1,2\}$, each connected component of $W^u(\C_{j}) \setminus \{\C_j\}$ coincides with a selected connected component of $W^s(\C_{(j+1)\,\text{mod}\,2}) \setminus \{\C_{(j+1)\,\text{mod}\,2}\}$.
\end{enumerate}

\bigskip

The two periodic solutions $\C_1$ and $\C_2$ and the set of trajectories referred to in (P3) build a heteroclinic cycle we will denote hereafter by $\mathcal{H}$. The assumptions (P1) and (P3) ensure that $\mathcal{H}$ is asymptotically stable (cf. \cite{KM1, KM2}), that is, there exists an open neighborhood $V^0$ of $\mathcal{H}$ in $\RR^3$ such that every solution starting in $V^0$ remains inside $V^0$ for all positive times and is forward asymptotic to $\mathcal{H}$. This open set $V^0$ is part of the basin of attraction of $\mathcal{H}$, which we denote by $\mathfrak{B}(\mathcal{H})$.

\medskip

Following the strategy adopted in \cite{Takens94, CR2017}, we will select cross sections (submanifolds of dimension two) inside linearizing neighborhoods of the periodic solutions (see Section~\ref{se:Local} for more details) and assume that, in appropriate coordinates, we have:
\medskip

\begin{enumerate}
\item[(\textbf{P4})] The transition maps are linear with diagonal and non-singular matrices given by
\tiny
$\left[ {\begin{array}{ccc}
1 & 0 & 0 \\
0 & a & 0 \\
0 & 0 & b \\
\end{array}}
\right]$
\normalsize
and
\tiny
 $\left[ {\begin{array}{ccc}
1 & 0 & 0 \\
0 & c & 0 \\
0 & 0 & d \\
\end{array}}
\right]$
\normalsize
with $a,\, c >0$, $0 < b,\, d \leq 1$.
\medskip

\item[(\textbf{P5})] The transition times between these cross sections are non-negative constants, say $s_1$ and $s_2$, not necessarily equal.
\medskip
\item[(\textbf{P6})]  The periodic solutions $\C_1$ and $\C_2$  have the same chirality. This means  that near $\C_1$ and $\C_2$ all solutions turn in the same direction around the two-dimensional connections $W^u(\C_1)$ and $W^u(\C_2)$. This is  a reformulation of the concept of \emph{similar chirality of two equilibria} proposed in Section 2.2 of \cite{LR2015}.
\end{enumerate}

\medskip

We denote by $\mathfrak{X}^r_{\text{PtoP}}(\EU^3)$ the set of $C^r$, $r\geq 2$, smooth vector fields in $\EU^3$ which satisfy the assumptions (P1)--(P6), endowed with the $C^r$-Whitney topology.

\section{Background material}\label{sse:definitions}

For the reader's convenience, we include in this section some definitions, notation and preliminary results.

\subsection{Invariants under conjugacy}
Given two vector fields $\dot{x} = f_1(x)$ and $\dot{x} = f_2(x)$, defined in domains $D_1\subset {\EU^3}$ and $D_2\subset {\EU^3}$, respectively, let $\varphi_i(t, x_0)$ be the unique solution of $\dot{x}=f_i(x)$ with initial condition $x(0)=x_0$, for $i \in \{1,2\}$. The corresponding flows are said to be \emph{topologically equivalent} in subregions $U_1 \subset D_1$ and $U_2 \subset D_2$ if there exists a homeomorphism $h: U_1 \rightarrow U_2$ which maps solutions of the first system onto solutions of the second preserving the time orientation. If $h$ is also time preserving, that is, if for every $x \in {\EU^3}$ and every $t \in \RR$, we have $\varphi_1(t, h(x))=h(\varphi_2(t,x))$, the flows are said to be \emph{topologically conjugate} and $h$ is called a \emph{topological conjugacy}.
A set of invariants under topological conjugacy is said to be \emph{complete} if, given two systems with equal invariants, there exists a topological conjugacy between the corresponding flows.

\subsection{Terminology}\label{sse:notation}

Given a compact, flow-invariant set $\mathcal{K} \subset \EU^3$, its \emph{basin of attraction} $\mathfrak{B}(\mathcal{K})$ is the set of points eventually attracted to $\mathcal{K}$, that is,
$$\mathfrak{B}(\mathcal{K}):= \Big\{x \in {\EU^3}  \colon \, \omega(x) \subset \mathcal{K} \Big\}$$
where $\omega(x)$ stands for the $\omega$-limit set of the trajectory of $x$.

We are especially interested in the case where $\mathcal{K}$ is a heteroclinic cycle. Let $\xi_1$ and $\xi_2$ be hyperbolic invariant sets. We say that there is a \emph{heteroclinic connection} from $\xi_1$ to $\xi_2$ if $W^u(\xi_1) \cap W^s(\xi_2) \neq \emptyset$. Note that this intersection may contain more than one trajectory and be of dimension greater than one. If there exist finitely many invariant hyperbolic sets $\xi_1, \ldots ,\xi_k$ and cyclic heteroclinic connections between them, namely $W^u(\xi_i) \cap W^s(\xi_{i+1}) \neq \emptyset$ for every $i \in \{1, \cdots, k-1\}$ and $W^u(\xi_k) \cap W^s(\xi_1) \neq \emptyset$, then the union of all sets and connections is called a \emph{heteroclinic cycle}. The sets $\xi_i$ may be equilibria, periodic solutions or more complicated invariant sets.

\subsection{Constants}\label{sse:constants}

For future use, we settle that:
$$\begin{array}{llll}
R_1 = \frac{\omega_1 \, \wp_1}{2\pi} \quad & R_2 = \frac{\omega_2 \,\wp_2}{2\pi} \quad & \gamma_1 = \frac{C_1}{E_2} \quad & \qquad \quad \quad \gamma_2 = \frac{C_2}{E_1}  \\ \\
\delta_1 = \frac{C_1}{E_1} \quad & \delta_2 = \frac{C_2}{E_2} \quad & \delta = \delta_1\,\delta_2 & \\ \\
\tau_1 = \frac{1}{E_1}\,(1+\gamma_1) \quad &  \tau_2 = \frac{1}{E_2}\,(1+\gamma_2). & & \\ \\
\end{array}$$

\noindent According to the assumptions, we have $\tau_1, \,\tau_2 > 0$, $\delta_1 > 1$ and $\delta_2 > 1$. Notice also that
\begin{eqnarray*}
\tau_1 &=& \frac{1}{E_1}\,(1+\gamma_1) = \frac{C_1 + E_2}{E_1\,E_2}\\
\tau_2 &=& \frac{1}{E_2}\,(1+\gamma_2) = \frac{E_1 + C_2}{E_1\,E_2} \\
\delta &=& \gamma_1\,\gamma_2 = \delta_1\,\delta_2 = \frac{C_1\,C_2}{E_1 \,E_2}.
\end{eqnarray*}

\section{Main result}\label{se:mainresults}

We now state the main theorem of this work. In Section~\ref{Example} we apply it to an example.

\begin{maintheorem}\label{teo:maintheorem-1}
Let $f\in \mathfrak{X}^r_{\text{\emph {PtoP}}}(\EU^3)$, $r\geq 2$. Then
$$\left\{\wp_1, \,\wp_2, \,\gamma_1, \,\gamma_2, \,\omega_1 + \gamma_1\omega_2, \, \omega_2 + \gamma_2 \omega_1, \,-\frac{1}{E_1} \log d + (s_1-\gamma_1 s_2), \,-\frac{1}{E_2} \log b + (s_2-\gamma_2 s_1)\right\}$$
is a complete set of invariants for $f$ under topological conjugacy in a neighborhood of the heteroclinic cycle $\mathcal{H}$.
\end{maintheorem}

The orbits of all points in the proper basin of attraction of $\mathcal{H}$ exhibit historic behavior, {a terminology introduced by Ruelle in \cite{R2001}.} This means that there exists a continuous map $G: \EU^3 \to \mathbb{R}$ whose sequence of Birkhoff time averages along each orbit in $\mathfrak{B}(\mathcal{H}) \setminus \mathcal{H}$ does not converge. Clearly, in the particular configuration of an attracting heteroclinic cycle between two periodic solutions $\C_1$ and $\C_2$, the $\omega$-limit of the orbits starting in $\mathfrak{B}(\mathcal{H}) \setminus \mathcal{H}$ includes the disjoint closed sets $\C_1$ and $\C_2$. In addition, the assumption (P1) on the values of $C_1, C_2, E_1$ and $E_2$ and the fact that the time these orbits spend near each one of the periodic solutions $\C_1$ and $\C_2$ is well distributed allow us to find such a map $G$. A proof of this fact may be read on the pages 1889-1891 of \cite{LR2017}.

\medskip

Observe that, if we assume that $s_1 = s_2 = 0$ (that is, both transitions are instantaneous), then the complete set of invariants reduces to
$$\left\{\wp_1, \,\wp_2, \,\gamma_1, \,\gamma_2, \,\omega_1 + \gamma_1\omega_2, \, \omega_2 + \gamma_2 \omega_1, -\frac{1}{E_1} \log d, \,-\frac{1}{E_2} \log b \right\}$$
a set which generalizes the ones found in \cite{Takens94} and \cite{CR2017}.

At the end of the paper the reader will gather convincing evidence that the essential steps of the proof of Theorem~\ref{teo:maintheorem-1} may be applied to attracting heteroclinic cycles between more than two hyperbolic periodic solutions, although the computations may be unwieldy. We conjecture that no qualitatively different invariant will arise within this more general setting. Regarding attracting homoclinic cycles associated to a periodic solution, see Section \ref{final}.

\section{Local and global dynamics in $\mathfrak{B}(\mathcal{H})$}\label{se:Local}

We will start defining two disjoint compact neighborhoods $V_1$ and $V_2$ of the $\C_{1}$ and $\C_{2}$, respectively, such that each boundary $\partial V_j$ is a finite union of smooth submanifolds (with boundary) which are transverse to the vector field.

\subsection{Local coordinates}\label{sse:Local-coodinates}
For $j \in \{1,2\}$, let $S_j$ be a cross section transverse to the flow at a point $P_j$ of $\C_{j}$. As $\C_{j}$ is hyperbolic, there is a neighborhood $\mathcal{V}^*_j$ of $P_j$ in $S_j$ where the first return map to $S_j$, denoted by $\pi_j$, is $C^1$ conjugate to its linear part (the eigenvalues of the derivative $D\pi_j(P_j)$ are precisely $e^{E_j}>1$ and $e^{-C_j}<1$). Moreover, for each $r\ge 2$ there is an open and dense subset of $\RR^2$ such that, if $C_j$ and $E_j$ lie in this set, then the conjugacy is of class $C^r$ (cf. \cite{Takens71}). The vector field associated to this linearization around $\C_j$ is represented by the system of differential equations given, in cylindrical coordinates $(\rho, \theta, z)$, by
\begin{equation}\label{ode of linearization}
\left\{
\begin{array}{l}
\dot{\rho}=-C_{j}\,(\rho - R_j) \\
\dot{\theta}=\omega_j \\
\dot{z}=E_{j}\,z
\end{array}
\right.
\end{equation}
\medskip
where $R_j=\frac{\omega_j \wp_j}{2\pi}$, whose solution with initial condition $(R_j + k, \,\theta_0, \,z_0)$, for $-\varepsilon \leq k \leq \varepsilon$, is
\begin{equation}\label{solutions of the ode of linearization}
t \in \mathbb{R} \quad \mapsto \quad \left\{
\begin{array}{l}
\rho(t) = R_j + k \, e^{-C_{j}\,t} \quad  \smallskip \\
\theta(t) = \theta_0 + \omega_j \,t  \quad  \text{mod } 2\pi. \\
z(t) = z_0 \, e^{E_{j}\,t}
\end{array}
\right.
\end{equation}
\medskip
and whose flow is $C^2$-conjugate to the flow of $f$ in a neighborhood of $\C_{j}$. Unless there is risk of misunderstanding, in what follows we will drop the label $\mod 2\pi$ when referring to the variable $\theta$. In these cylindrical coordinates,
\begin{itemize}
\item[(a)] the periodic solution $\C_{j}$ is the circle described by $\rho = R_j$ and $z=0$;
\medskip
\item[(b)] the local stable manifold $W^s_{\loc}(\C_{j})$ of $\C_{j}$ is the plane defined by $z=0$;
\medskip
\item[(c)] the local unstable manifold $W^u_{\loc}(\C_{j})$ of $\C_{j}$ is the cylindrical surface defined by $\rho=R_j$.
\end{itemize}
See the illustration in Figure~\ref{local_C}.

\begin{figure}[h]
\begin{center}
\includegraphics[height=10cm]{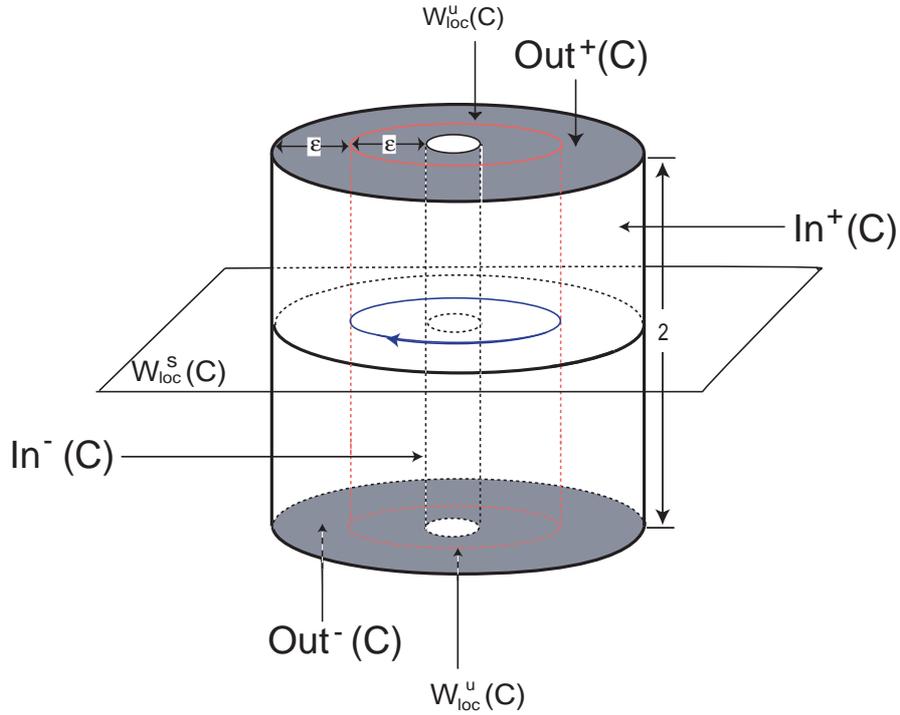}
\end{center}
\caption{\small Local data near a periodic solution $\C$.}
\label{local_C}
\end{figure}

\bigskip

We will analyze the dynamics inside a cylindrical neighborhood $V_j(\varepsilon)$ of $\C_{j}$, for some $\varepsilon > 0$, contained in the saturation of $\mathcal{V}^*_j$ by the flow and given by
$$V_j(\varepsilon) = \Big\{(\rho,\theta,z) \colon \quad 0 < R_j-\varepsilon \le \rho \le R_j+\varepsilon, \quad \theta \in [0, 2\pi[, \quad -\varepsilon \le z \le \varepsilon\,\Big\}.$$
When there is no risk of confusion, we will write $V_j$ instead of $V_j(\varepsilon)$. For $j \in \{1,2\}$, each $V_j$, called an \emph{isolating block} for $\C_{j}$, is homeomorphic to a hollow cylinder whose boundary is the union $\partial V_{j}= \Inn(\C_{j}) \cup \Out(\C_{j}) \cup \Delta(\C_{j})$ satisfying the following conditions:

\begin{itemize}
\item[(1)] $\Inn(\C_{j})$ is the union of the walls of $V_j$, that is,
$$\Inn(\C_{j}) = \Big\{(\rho, \theta, z)\colon \, \rho = R_j \pm \varepsilon, \quad \theta \in [0, 2\pi[, \quad |z| \le \varepsilon\Big\}$$
with two connected components which are locally separated by $W^u(\C_{j})$.
In cylindrical coordinates, $\Inn(\C_{j})\cap W^s(\C_{j})$ is the union of the two circles in $V_j$, namely
$$\Inn(\C_{j})\cap W^s(\C_{j}) = \Big\{(\rho, \theta, z)\colon \, \rho = R_j \pm \varepsilon, \quad \theta \in [0, 2\pi[, \quad z=0 \Big\}.$$
Forward trajectories starting at $\Inn(\C_{j})$ go inside $V_j$.
\bigskip
\item[(2)] $\Out(\C_{j})$ is the union of two annuli, the top and the bottom of $V_j$, that is,
$$\Out(\C_{j}) = \Big\{(\rho, \theta, z)\colon \, R_j - \varepsilon \le \rho \le R_j + \varepsilon, \quad \theta \in [0, 2\pi[, \quad z = \pm \varepsilon\Big\}$$
with two connected components which are locally separated by $W^s(\C_{j})$. The intersection $\Out(\C_{j})\cap W^u(\C_{j})$ is precisely the union of the two circles in $V_j$ given by
$$\Out(\C_{j})\cap W^u(\C_{j}) = \Big\{(\rho, \theta, z)\colon \, \rho=R_j, \quad \theta \in [0, 2\pi[, \quad z=\pm \varepsilon \Big\}.$$
Backward trajectories starting at $\Out(\C_{j})$ go inside $V_j$.
\bigskip
\item[(3)] The vector field is transverse to $\partial V_{j}$ at all points except possibly at the circles $\Delta(\C_{j})=\overline{\Inn(\C_{j})}\cap \overline{\Out(\C_{j})}$, parameterized by $\rho=R_j \pm \varepsilon$ and $z=\pm \varepsilon$.
\medskip
\end{itemize}

\medskip

Denote by $\Inn^+(\C_j)$ the intersection of $\Inn(\C_j)$ with $ \rho = R_j + \varepsilon$, and let $\Out^+(\C_j)$ be the intersection of $\Out(\C_j)$ with $z=\varepsilon$. More precisely,
\begin{eqnarray}
\Inn^+(\C_j) &=& \Big\{(\rho, \theta, z)\colon \, \rho = R_j + \varepsilon, \quad \theta \in [0, 2\pi[, \quad - \varepsilon \leq z \leq \varepsilon\Big\} \medskip \\
\Out^+({\C_j})&=& \Big\{(\rho, \theta, z)\colon \, R_j - \varepsilon \le \rho \le R_j + \varepsilon, \quad \theta \in [0, 2\pi[, \quad z = \varepsilon \Big\}. \nonumber
\end{eqnarray}

\subsection{Local dynamics}\label{sse:Local-dynamics}
In this subsection we restrict the analysis to initial points of $\Inn (\C_j)$ with $z_0>0$ and $\rho=R_j+\varepsilon$.  The other cases are entirely similar.
Using the dynamics in local coordinates described by \eqref{solutions of the ode of linearization}, we now evaluate the time needed by an initial condition $(R_j+ \varepsilon, \,\theta_0, \,z_0) \in \Inn^+(\C_{j})$ to reach $\Out^+(\C_{j})$.
\medbreak

To estimate this time $T$, we have just to solve the equation
$$z_0 \, e^{E_{j}\,T} = \varepsilon$$
from which we deduce that
$$T= -\frac{1}{E_j}\,\log\,\left(\frac{z_0}{\varepsilon}\right).$$
Therefore, the local map, acting inside $V_j$ and sending $\Inn^+(\C_{j})$ into $\Out(\C_{j})$, is given by
\begin{eqnarray}\label{local map}
\Phi _{j}^+(R_j + \varepsilon, \theta_0, z_0) &=& \left(\rho(T), \, \theta(T), \, z(T)\right) \\
&=& \left(R_j + \varepsilon \left(\frac{z_0}{\varepsilon}\right)^{\delta_j}, \,\,\, \theta_0-\frac{\omega_j}{E_j}\,\log \,\left(\frac{z_0}{\varepsilon}\right) \,\, \text{mod } 2\pi,\,\,\, \varepsilon\right). \nonumber
\end{eqnarray}

\subsection{Transition maps}\label{subglobal}

Denote by $[\C_1\rightarrow\C_2]$ the component of the heteroclinic cycle $\mathcal{H}$ formed by the coincidence between $W^u(\C_1)$ and $W^s(\C_2)$. Similarly, $[\C_2\rightarrow\C_1]$ represents the coincidence between $W^s(\C_1)$ and $W^u(\C_2)$. Notice that $[\C_1\rightarrow \C_2]$ connects points with $z=\varepsilon$ in $V_{1}$ (respectively $z=-\varepsilon$) to points with $\rho=R_2+\varepsilon$ (respectively $\rho=R_2-\varepsilon$) in $V_{2}$.

Notice that $\Out^+(\C_1) \setminus [\C_1 \rightarrow \C_2]$ has two connected components (the same holds for $\Out^+(\C_2)$) and that points in $\Out^+(\C_1)$ near $W^u(\C_1)$ are mapped into $\Inn^+(\C_2)$ along a flow-box around the connection $[\C_1\rightarrow\C_2]$; analogously, points in $\Out^+(\C_2)$ near $W^u(\C_2)$ are mapped into $\Inn^+(\C_1)$ along the same flow-box.

Recall that, by Property (P4), we are assuming that both transition maps from $\Out^\pm (\C_j)$ to $\Inn^\pm(\C_j)$, for $j=1,2$, have a linear component with submatrices
\tiny
$\left[ {\begin{array}{cc}
   a & 0 \\
   0 & b \\
  \end{array} }
\right]$
\normalsize
from $\Out (\C_1)$ to $\Inn (\C_2)$, and
\tiny
$\left[ {\begin{array}{cc}
   c & 0 \\
   0 & d \\
  \end{array} } \right]$
\normalsize
from $\Out (\C_2)$ to $\Inn(\C_1)$, for some $0 < b,\,d \leq 1$ and $a,\,c >0$. Therefore, the transition maps $\Psi_{12}^+ \colon \,\Out^+(\C_1) \,\,\to\,\, \Inn^+(\C_2)$ and $\Psi_{21}^+ \colon \,\Out^+(\C_2) \,\,\to \,\,\Inn^+(\C_1)$ are expressed in cylindrical coordinates as
\begin{eqnarray}
\Psi_{12}^+(\rho, \theta, \varepsilon) &=& \Big(R_2 + \varepsilon,\,\,\, a\,\theta \,\,\, \text{mod }2\pi, \,\,\,b\left(\rho-R_1\right)\Big)
\end{eqnarray}
and
\begin{eqnarray}
\Psi_{21}^+(\rho, \theta, \varepsilon) &=& \Big(R_1 + \varepsilon, \,\,\,c\,\theta \,\,\, \text{mod }2\pi, \,\,\, d\left(\rho-R_2\right)\Big).
\end{eqnarray}
Figure~\ref{matrizes1} summarizes this information.

\begin{figure}[h]
\begin{center}
\includegraphics[height=8cm]{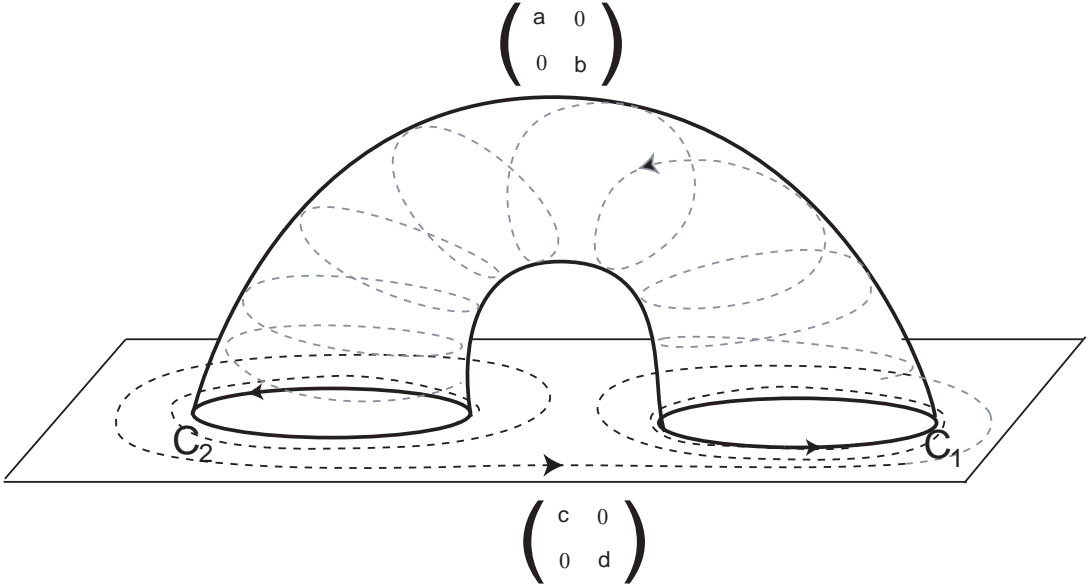}
\end{center}
\caption{\small Linear components of the global maps.}
\label{matrizes1}
\end{figure}

\medskip

\subsection{The first return map to $\Inn(\C_2)$}\label{sse:1st return map}

Given an initial condition $(R_2\pm \varepsilon, \theta, z) \in \Inn^+(\C_2)$, its trajectory returns to $\Inn^+(\C_2)$, thus defining a first return map
\begin{equation}\label{first return}
\mathcal{F}_2 := \Psi_{12}^+\circ\Phi_1^+ \circ \Psi_{21}^+\circ\Phi_2^+ \colon \quad  \Inn^+(\C_2) \,\, \rightarrow \,\, \Inn^+(\C_2)
\end{equation}
which is as smooth as the vector field $f$ and acts as\\
\begin{equation}
\label{first_return1}
\mathcal{F}_2(R_2 \pm \varepsilon, \theta, z) = \left(R_2 \pm \varepsilon, \Theta, \, Z \right),
\end{equation}
where
\begin{eqnarray*}
\Theta &=&  {ac}\, \theta - \left[\frac{a c\,\, \omega_1 E_1 + a\,\,\omega_1 \, C_2}{E_1\,E_2}\right] \log \,\left(\frac{z}{\varepsilon}\right) - \frac{a\, \omega_1}{E_1}\log d \qquad \text{mod} \, 2\pi\\
Z &=&  b\,\, \varepsilon \, \, d^{\delta_1}\left(\frac{z}{\varepsilon}\right)^\delta.
\end{eqnarray*}

\medskip

If $s_1(X)$ stands for the time needed for the orbit starting at $X \in \Out(\C_2)$ to hit $\Inn(\C_1)$ (see Figure~\ref{global}) and we choose the cross sections $\Out(\C_2)$ and $\Inn(\C_1)$ small enough, then the interval $[s_{\text{min}}, s_{\text{max}}]$ is arbitrarily small, where
\begin{eqnarray*}
s_{\text{min}} &=& \min \, \Big\{s_1(X) \colon \, \,  X \in \Out(\C_2)\cap W^u(\C_2)\Big\} \\
s_{\text{max}} &=& \max \, \Big\{s_1(X) \colon \, \,  X \in \Out(\C_2)\cap W^u(\C_2)\Big\}.
\end{eqnarray*}
Notice that these extreme values exist since $\Out(\C_2)\cap W^u(\C_2)$ is compact. Therefore, there is $M_1>0$ such that $0\leq s_1 (X)\leq M_1$ for all $X \in \Out(\C_2)$. Analogously, we define $s_2(X)$ as the time needed for the orbit starting at $X \in \Out(\C_1)$ to hit $\Inn(\C_2)$. Using the same argument, we may find  $M_2>0$ such that $0 \leq s_2(X) \leq M_2$ for all $X \in \Out(\C_1)$. Let $M =\max\, \{M_1, M_2\}$. We remark that, for each initial condition $X_0 \in \mathfrak{B}(\mathcal{H})$, the time spent by the piece of the trajectory $\{\varphi(t,\,X_0)\colon \, t \in \,[0,\mathrm{T}]\}$ inside $V_1 \cup V_2$ goes to infinity as $\mathrm{T} \to +\infty$, while both transition times $s_1$ and $s_2$ during its sojourn outside $V_1 \cup V_2$ remain uniformly bounded.

\begin{figure}[h]
\begin{center}
\includegraphics[height=7cm]{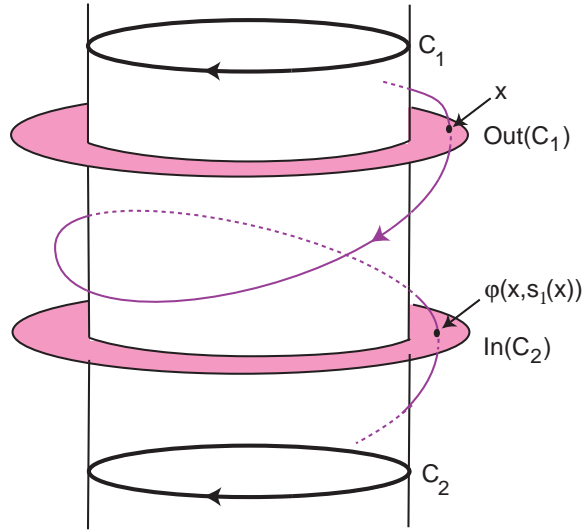}
\end{center}
\caption{\small Scheme for the global transition.}
\label{global}
\end{figure}

\section{Hitting times}\label{se:hitting-times}

In this section we will obtain estimates of the amount of time a trajectory spends between consecutive isolating neighborhoods of the periodic solutions. To simplify the computations, we may re-scale the local coordinates in order to assume that $\varepsilon=1$.

As a trajectory approaches $\mathcal{H}$, it visits a neighborhood of $\C_1$, then moves off towards a neighborhood of $\C_2$, comes back to the proximity of $\C_1$, and so on. During each turn it spends a geometrically increasing period of time in the small neighborhoods of the periodic solutions. More precisely, starting at the time $t_0$ (which we may assume equal to $0$) with the initial condition $(\rho_0, \theta_0, 1) \in \Out^+(\C_2)$, its orbit hits $\Out^+(\C_1)$ after a time interval equal to
\begin{equation}\label{eq:t1}
t_1 = s_1(\rho_0, \theta_0, 1) - \,\frac{1}{E_1} \log \,\left(d\,|\rho_0-R_2|\right)
\end{equation}
at the point in $\Out^+(\C_1)$ whose cylindrical coordinates are
\begin{eqnarray*}
(\rho_1, \, \theta_1, \,1) &=& (\Phi^+_{1} \circ \Psi^+_{2  1}) (\rho_0,\,\theta_0, \,1) = \Phi^+_{1} \left(R_1 + 1, \,\, c\,\theta_0, \,\, d\left(\rho_0 - R_2\right)\right) \\ \\
&=& \left(R_1 + \left[d(\rho_0-R_2)\right]^{\delta_1},\,\,c\,\theta_0 -\frac{\omega_1}{E_1}\log \left[d(\rho_0-R_2)\right],\,\,1\right) \quad \quad \text{if $\,\,\rho_0 > R_2$}; \\ \\
(\rho_1, \, \theta_1, \,1) &=& (\Phi^+_{1} \circ \Psi^+_{2  1}) (\rho_0,\,\theta_0, \,1) = \Phi^+_{1} \left(R_1 - 1, \,\, c\,\theta_0, \,\, d\left(R_2 - \rho_0\right)\right) \\ \\
&=& \left(R_1 - \left[d(R_2 -\rho_0)\right]^{\delta_1},\,\,c\,\theta_0 -\frac{\omega_1}{E_1}\log \left[d(R_2 - \rho_0)\right],\,\,1\right) \quad \quad \text{if $\,\,\rho_0 < R_2$}.
\end{eqnarray*}
Then, the orbit goes to $\Inn^+(\C_2)$ and proceeds to $\Out^+(\C_2)$, hitting the point
$$(\rho_2, \, \theta_2, \,1) = (\Phi^+_{2}\circ \Psi^+_{12})(\rho_1, \, \theta_1, \,1) $$
in $\Out^+(\C_2)$, where
\begin{eqnarray*}
\rho_2 &=& R_2 \pm b^{\delta_2}\,\left[d|\rho_0-R_2|\right]^\delta, \\
\theta_2 &=& {ac}\, \theta_0 - \left[\frac{a \, \omega_1\, E_2 + \omega_2 \, C_1}{E_1\,E_2}\right] \log \,{|\rho_0-R_2|}  -  \left[\frac{a \,\omega_1 E_2 + \omega_2 \, C_1}{E_1\,E_2}\right] \log d - \frac{\omega_2}{E_2}\log b \quad \text{mod} \, 2\pi,\\
\end{eqnarray*}
and spending in the whole path a time equal to
\begin{eqnarray}\label{eq:t2}
t_2 &=& t_1 + s_2(\rho_1, \, \theta_1, \,1) + \left(-\,\frac{1}{E_2} \log \,\left(b\,|\rho_1-R_1|\right)\right) \\
&=& t_1 + s_2(\rho_1, \, \theta_1, \,1) -\, \frac{1}{E_2}\log b - \, \frac{\delta_1}{E_2} \log d - \, \frac{\delta_1}{E_2} \log \,(|\rho_0-R_2|) . \nonumber
\end{eqnarray}
And so on for the other time values.

\section{The invariants}\label{se:invariants}

Now we will examine how the hitting times sequences generate the set of invariants we are looking for. Starting with a point $P_0:=(\rho_0, \theta_0,1) \in \Out^+(\C_2)$ at the time $t_0=0$ (notice that $P_0 \in \mathcal{B(H)}\backslash \mathcal{H}$), we consider the sequences of times $\left(t_j\right)_{j\,\in\,\mathbb{N}}$ constructed in the previous section and define, for each $i \in \NN_0=\NN\cup\{0\}$, the sequences of points and transition times\\

\begin{equation}\label{eq:times}
\left\{\begin{array}{l}
P_{2i}:=\varphi\left(t_{2i} , \,\,P_0\right)=(\rho_{2i},\,\,\theta_{2i}, \,\,1) \quad \in\,\, \Out^+(\C_2) \medskip \\
s_{2i+1}:=s_{2i+1}(P_0)= s_1(P_{2i}) \medskip \\
P_{2i+1}:=\varphi\left(t_{2i+1}, \,\,P_0\right)=(\rho_{2i+1},\,\,\theta_{2i+1}, \,\,1) \quad \in\,\, \Out^+(\C_1) \medskip \\
s_{2i+2}= s_{2i+2}(P_1)= s_2(P_{2i+1}).
\end{array}
\right.
\end{equation}
\bigskip

The trajectory $(t \,\in\, \RR^+_0 \,\to\, \varphi(t,P_0))$ is partitioned into periods of time corresponding either to its sojourns inside $V_{1}$ and along the connection $[\C_2 \rightarrow \C_1]$  (that is, the differences ${t_{2i+1}-t_{2i}}$ for $i \in \NN_0$) or inside $V_{2}$ and along the the connection $[\C_1\rightarrow \C_2]$ (that is, ${t_{2i+2}-t_{2i+1}}$ for $i \in \NN_0$) during its travel that begins and ends at $\Out^+(\C_2)$.

\begin{lemma}\label{le:calculus} Let $P_0=(\rho_0, \theta_0,1)$ be a point in $\Out^+(\C_2)$ and take the corresponding sequence $(t_j)_{j\,\in\, \NN_0}$. Then:
\begin{enumerate}
\item $(t_{2i+1}-t_{2i})-\gamma_2\,(t_{2i}-t_{2i-1}) = -\,\frac{1}{E_1}\log d + (s_{2i+1} - \gamma_2\,s_{2i}).$
\smallbreak
\item $(t_{2i+2}-t_{2i+1})-\gamma_1\,(t_{2i+1}-t_{2i}) = -\,\frac{1}{E_2}\log b + (s_{2i+2} - \gamma_1\,s_{2i+1}).$
\smallbreak
\item $(t_{2i+2}-t_{2i})- \delta\,(t_{2i}-t_{2i-2}) = -\tau_1\log d  -\tau_2 \log b + (s_{2i+2} + s_{2i+1}) - \delta \,(s_{2i} + s_{2i-1})$.
\medbreak
\end{enumerate}
\end{lemma}

\begin{proof}
Firstly, recall from \eqref{eq:t1} and \eqref{eq:t2} that
\begin{eqnarray*}
t_{2i} - t_{2i-1} &=& -\frac{1}{E_2} \log\,\left(b\, |\rho_{2i-1} - R_1|\right) + s_{2i}\\
t_{2i+1} - t_{2i} &=& -\frac{1}{E_1} \log\,\left(d\, |\rho_{2i} - R_2|\right) + s_{2i+1}.
\end{eqnarray*}
Besides, one has
\begin{eqnarray*}
t_{2i+1} - t_{2i} &=& -\frac{1}{E_1} \log\,\left(d\, |\rho_{2i} - R_2|\right) + s_{2i+1} = -\frac{1}{E_1} \log\,\left[d\, \Big(b\,|\rho_{2i-1}-R_1|\Big)^{\delta_2}\right] + s_{2i+1}\\
&=& -\frac{1}{E_1} \log d - \frac{\delta_2}{E_1} \log b - \frac{\delta_2}{E_1} \log\,\left( |\rho_{2i-1}-R_1|\right) + s_{2i+1}.
\end{eqnarray*}
Therefore,
\begin{eqnarray*}
&& (t_{2i+1}-t_{2i})-\gamma_2\,(t_{2i}-t_{2i-1}) = (t_{2i+1}-t_{2i})-\frac{C_2}{E_1}\,({t_{2i}-t_{2i-1}}) \\
&=& -\frac{1}{E_1} \log d - \frac{\delta_2}{E_1} \log b - \frac{\delta_2}{E_1} \log\,\left( |\rho_{2i-1}-R_1|\right) + s_{2i+1} - \\
&-& \frac{C_2}{E_1}\, \left[-\frac{1}{E_2} \log b - \frac{1}{E_2} \log\,\left(|\rho_{2i-1} - R_1|\right) + s_{2i}\right] \\
&=& -\frac{1}{E_1} \log d - \frac{\delta_2}{E_1} \log b  + \frac{C_2}{E_1}\,\frac{1}{E_2} \log b   + \Big(s_{2i+1} - \frac{C_2}{E_1}\,s_{2i}\Big) \\
&=& -\frac{1}{E_1} \log d + \Big(s_{2i+1} - \frac{C_2}{E_1}\,s_{2i}\Big).
\end{eqnarray*}
The proof of item (2) of the lemma is similar. Concerning item (3), we start evaluating $t_{2i}-t_{2i-2}$ and $t_{2i+2}-t_{2i}$:
\begin{eqnarray*}
t_{2i}-t_{2i-2} &=&  -\frac{1}{E_2} \log \,(b\, |\rho_{2i-1} - R_1|) + s_{2i-1} - \frac{1}{E_1} \log \,(d\, |\rho_{2i-2} - R_2|) + s_{2i}\\
&=&  -\frac{1}{E_2} \log \,\left[b \, \Big(d \, |\rho_{2i-2}-R_2|\Big)^{\delta_1}\right] - \frac{1}{E_1} \log \,(d\, |\rho_{2i-2} - R_2|) + \Big(s_{2i} + s_{2i-1}\Big)\\
&=&  -\left(\frac{1}{E_1}  + \frac{\delta_1}{E_2} \right) \log d  -\frac{1}{E_2} \log b  -\left(\frac{1}{E_1}  + \frac{\delta_1}{E_2} \right) \log \,(|\rho_{2i-2}-R_2|) + \Big(s_{2i} + s_{2i-1}\Big) \\
&=& -\tau_1 \log d  -\frac{1}{E_2} \log b  -\tau_1 \log \,(|\rho_{2i-2}-R_2|) + \Big(s_{2i} + s_{2i-1}\Big);
\end{eqnarray*}
\medskip
\begin{eqnarray*}
t_{2i+2}-t_{2i} &=& -\tau_1 \log d  -\frac{1}{E_2} \log b  - \tau_1 \log \,(|\rho_{2i}-R_2|) + \Big(s_{2i+2} + s_{2i+1}\Big)\\
&=& -\tau_1 \log d  -\frac{1}{E_2} \log b  -\tau_1 \log \,\left[\left(b\, (d\, |\rho_{2i-2}-R_2|)^{\delta_1}\right)^{\delta_2}\right] + \Big(s_{2i+2} + s_{2i+1}\Big) \\
&=& -\tau_1 \log d  -\frac{1}{E_2} \log b  -\tau_1\delta_2 \log \,\left[b\, (d\, |\rho_{2i-2}-R_2|)^{\delta_1}\right] + \Big(s_{2i+2} + s_{2i+1}\Big) \\
&=& -\tau_1 \log d  - \left(\frac{1}{E_2} + \tau_1\delta_2\right) \log b - \tau_1\delta_1\delta_2 \log \,\left(d\, |\rho_{2i-2}-R_2|\right) + \Big(s_{2i+2} + s_{2i+1}\Big)\\
&=&  -\tau_1(1+\delta) \log d  - \left(\frac{1}{E_2} + \tau_1\delta_2\right) \log b  - \tau_1\delta \log \,\left(|\rho_{2i-2}-R_2|\right) + \Big(s_{2i+2} + s_{2i+1}\Big).
\end{eqnarray*}
\bigskip
Finally, combining the two previous equalities, we obtain
\begin{eqnarray*}
&& (t_{2i+2}-t_{2i})- \delta\,(t_{2i}-t_{2i-2}) = \\
&=&  -\tau_1(1+\delta) \log d  -\left(\frac{1}{E_2} + \tau_1\delta_2\right) \log b  - \tau_1\delta \log \,\left(|\rho_{2i-2}-R_2|\right) \\
&& \quad \quad \quad + \,\,\tau_1\delta \log d  + \frac{\delta}{E_2} \log b  + \tau_1\delta \log \,(|\rho_{2i-2}-R_2|) + \Big(s_{2i+2} + s_{2i+1}\Big) - \delta \,\Big(s_{2i} + s_{2i-1}\Big)\\
&=&  -\tau_1\log d  - \left(\frac{1}{E_2} + \tau_1\delta_2 - \frac{\delta}{E_2}\right) \log b  + \Big(s_{2i+2} + s_{2i+1}\Big) - \delta \,\Big(s_{2i} + s_{2i-1}\Big) \\
&=&  -\tau_1\log d - \frac{1}{E_2} \left(1+\gamma_2\right) \log b + \Big(s_{2i+2} + s_{2i+1}\Big) - \delta \,\Big(s_{2i} + s_{2i-1}\Big) \\
&=&  -\tau_1\log d  -\tau_2 \log b + \Big(s_{2i+2} + s_{2i+1}\Big) - \delta \,\Big(s_{2i} + s_{2i-1}\Big).
\end{eqnarray*}
\end{proof}

Taking into account that the sequences $(s_{2i})_{i\in \NN}$ and $(s_{2i-1})_{i\in \NN}$ are uniformly bounded, a straightforward computation gives additional information on the evolution of the quotients of the previous sequences, besides a connection between the return times sequences and the combinations
$\omega_1 + \gamma_1\,\omega_2$ and $\omega_2 + \gamma_2\,\omega_1$.

\begin{corollary}\label{cor:speed-of-convergence} $\,$
\begin{enumerate}
\item $\lim_{i \to +\infty} \,\frac{t_{2i+2}\,-\,t_{2i+1}}{t_{2i+1}\,-\,t_{2i}} = \gamma_1$.
\medskip
\item $\lim_{i \to +\infty} \,\frac{t_{2i+1}\,-\,t_{2i}}{t_{2i}\,-\,t_{2i-1}} = \gamma_2$.
\medskip
\item $\lim_{i \to +\infty} \,\frac{t_{2i+2}\,-\,t_{2i}}{t_{2i}\,-\,t_{2i-2}} = \delta$.
\medskip
\item $\lim_{i \to +\infty}\, \frac{\omega_1\,\left(t_{2i+1}\,-\,t_{2i}\right) \,+ \,\omega_2\,\left(t_{2i+2}\,-\,t_{2i+1}\right)}{t_{2i+2}\,-\,t_{2i}} = \left(\omega_1 + \gamma_1\,\omega_2\right)\,\frac{1}{\gamma_1 \,+\, 1}.$
\medskip
\item $\lim_{i \to +\infty}\, \frac{\omega_2\,\left(t_{2i}\,-\,t_{2i-1}\right) \,+ \,\omega_1\,\left(t_{2i+1}\,-\,t_{2i}\right)}{t_{2i+1}\,-\,t_{2i-1}} = \left(\omega_2 + \gamma_2\,\omega_1\right)\,\frac{1}{\gamma_2 \,+\, 1}.$
\end{enumerate}
\end{corollary}

\bigskip

Observe that
$$\left(\omega_1 + \gamma_1\,\omega_2\right)\,\frac{1}{\gamma_1 \,+\, 1} - \left(\omega_2 + \gamma_2\,\omega_1\right)\,\frac{1}{\gamma_2 \,+\, 1} = \left(\omega_1 - \omega_2\right)\,\frac{1-\gamma_1\,\gamma_2}{(\gamma_1 \,+\,1)(\gamma_2 \,+\, 1)}$$
so, under assumption (P1), the invariants $\left(\omega_1 + \gamma_1\,\omega_2\right)\,\frac{1}{\gamma_1 \,+\, 1}$ and $\left(\omega_2 + \gamma_2\,\omega_1\right)\,\frac{1}{\gamma_2 \,+\, 1}$ are equal if and only if $\omega_1 = \omega_2$.

\medskip

From now on, and having in mind the assumption (P5) and the examples we are interested in (see Section \ref{Example}), we will assume that there exist $s_1\geq 0$ and $s_2\geq 0$ such that
\begin{equation}
\label{times1}
s_{2i+1}= s_1 \qquad \text{and} \qquad s_{2i}= s_2, \quad \quad \forall \,i\,\in \,\NN \quad \quad \forall\, P_0 \, \in \,\Out^+(\C_2).
\end{equation}
This way, using the previous computations, we may estimate the invariants we are looking for.

\begin{corollary}\label{cor:invariants2} $\,$
Let $P_0=(\rho_0, \theta_0,1)$ be a point in $\Out^+(\C_2)$ and take the corresponding times sequence $(t_i)_{i\in\, \NN_0}$. Then:
\begin{enumerate}
\item $\lim_{i \to +\infty} \, (t_{2i+1}-t_{2i})-\gamma_2\,(t_{2i}-t_{2i-1}) = -\,\frac{1}{E_1}\log d + (s_{1} - \gamma_2\,s_{2}).$
\smallbreak
\item $\lim_{i \to +\infty} \, (t_{2i+2}-t_{2i+1})-\gamma_1\,(t_{2i+1}-t_{2i}) = -\,\frac{1}{E_2}\log b + (s_{2} - \gamma_1\,s_{1}).$
\smallbreak
\item $\lim_{i \to +\infty} \, (t_{2i+2}-t_{2i})- \delta\,(t_{2i}-t_{2i-2}) = -\tau_1\log d  -\tau_2 \log b + (s_{2} + s_{1}) (1-\delta)$.
\smallbreak
\end{enumerate}
\end{corollary}

\medskip

Thus, besides $\wp_1$, $\wp_2$, the values
\begin{eqnarray*}
\gamma_1 &\quad& \gamma_2 \\
\omega_1 + \gamma_1\omega_2 &\quad& \omega_2 + \gamma_2 \omega_1 \\
-\frac{1}{E_1} \log d + (s_1-\gamma_1 s_2) &\quad& -\frac{1}{E_2} \log b + (s_2-\gamma_2 s_1)
\end{eqnarray*}
are invariants under topological conjugacy. Notice that the invariant
$$-\tau_1 \log d  -\tau_2 \log b + (s_1 + s_2)(1-\delta)$$
may be rewritten as a combination of $-\,\frac{1}{E_1}\log d + (s_1 - \gamma_2\,s_2)$ and $-\,\frac{1}{E_2}\log b + (s_2 - \gamma_1\,s_1)$ with coefficients that are invariants as well. Indeed, summoning the links between the several constants listed in Subsection~\ref{sse:constants}, we deduce that
\begin{eqnarray*}
&&\left[-\,\frac{1}{E_1}\log d + (s_{1} - \gamma_2\,s_{2}) \right](1 + \gamma_1) + \left[-\,\frac{1}{E_2}\log b + (s_{2} - \gamma_1\,s_{1}) \right](1 + \gamma_2) \\
&=& -\frac{1}{E_1} \log d + s_1 - \gamma_2 \,s_2 -\frac{\gamma_1}{E_1} \log d + \gamma_1 \,s_1 -\gamma_1\,\gamma_2\, s_2 -\frac{1}{E_2} \log b + s_2  -\gamma_1\, s_1\\
&&  - \frac{\gamma_2}{E_2} \log b + \gamma_2\, s_2 - \gamma_1\,\gamma_2\, s_1 \\
&=& \Big(-\frac{1 + \gamma_1}{E_1}\Big) \log d + \Big(- \frac{1+\gamma_2}{E_2}\Big) \log b + \Big(s_1 + s_2\Big)\Big(1 - \gamma_1\,\gamma_2\Big) \\
&=& -\tau_1 \log d  - \tau_2 \log b + (s_1 + s_2)(1-\delta).
\end{eqnarray*}

\section{Completeness of the set of invariants}\label{se:completeness}

Let $f$ and $g$ be vector fields in $\mathfrak{X}^r_{\text{PtoP}}(\EU^3)$, $r \geq 2$, having a stable heteroclinic cycle associated to two periodic solutions. For a conjugacy between $f$ and $g$ to exist it is necessary that the conjugated orbits have hitting times sequences, with respect to fixed cross sections, that are uniformly close. Therefore, besides the numbers $\wp_1$ and $\wp_2$, which are well known to be invariants under conjugacy, the values $\gamma_1$, $\gamma_2$, $-\frac{1}{E_1} \,\log d + (s_1-\gamma_2 s_2)$, $-\frac{1}{E_2} \,\log b + (s_2-\gamma_1 s_1)$,  $\omega_1 + \gamma_1 \,\omega_2$ and $\omega_2 + \gamma_2 \,\omega_1$ are also invariants under topological conjugacy. We are left to prove that they form a complete set. The argument we will present was introduced by F. Takens in \cite{Takens94} while examining Bowen's example and, with some adjustments, used in \cite{CR2017} for a class of Bykov attractors.

Let $\wp_1$, $\,\wp_2$, $\,\gamma_1$, $\,\gamma_2$, $\,\omega_2 + \gamma_2 \omega_1$, $\,\omega_1 + \gamma_1\omega_2$, $\,-\frac{1}{E_1} \,\log d + (s_1-\gamma_2 s_2)$ and $\, - \frac{1}{E_2} \,\log b + (s_2-\gamma_1 s_1)$ be the invariants of $f$, and $\overline{\wp}_1$, $\,\overline{\wp}_2$, $\,\overline{\gamma}_1$, $\,\overline{\gamma}_2$,
$\,\overline{\omega}_1 + \overline{\gamma_1}\overline{\omega}_2$,
$\,\overline{\omega}_2 + \overline{\gamma}_2\overline{\omega}_1$, $\,-\frac{1}{\overline{E}_1} \,\log \overline{d} + (\overline{s}_1-\overline{\gamma}_2 \overline{s}_2)$ and $\, -\frac{1}{\overline{E}_2} \,\log \overline{b} + (\overline{s}_2-\overline{\gamma}_1 \overline{s}_1)$ the ones of $g$. Assume that they are pairwise equal. We are due to explain how these numbers enable us to construct a conjugacy between $f$ and $g$ in a neighborhood of the respective heteroclinic cycles $\mathcal{H}_f$ and $\mathcal{H}_g$.

\subsection{Takens' argument} We will start associating to $f$ and any point $P$ in a fixed cross section $\Sigma$  another point $\widetilde{P}$ whose $f-$trajectory has a sequence of hitting times (at a possibly different but close cross section $\widetilde{\Sigma}$) which is determined by, and uniformly close to, the hitting times sequence of $P$, but is easier to work with. This is done by slightly adjusting the cross section $\Sigma$ using the flow along the orbit of $P$. Afterwards, we need to find an injective and continuous way of recovering the orbits from the hitting times sequences. Repeating this procedure with $g$ we find a point $Q$ whose $g-$trajectory has hitting times at some cross section equal to the ones of $\widetilde{P}$. Due to the fact that the invariants of $f$ and $g$ are the same, the map that sends $P$ to $Q$ is the desired conjugacy.

\subsection{A sequence of adjusted hitting times}\label{sse:time2}
Fix $P = (\rho_0, \,\theta_0,\,z_0) \in \mathfrak{B}(\mathcal{H}_f)$ and let $\left(t_i\right)_{i \,\in\, \NN_0}$ be the times sequence defined in \eqref{eq:times}. We start defining, for each $i \in \NN_0$, a finite family of numbers
$$\widetilde{T}_0^{\,\,(i)},\,\, \widetilde{T}_1^{\,\,(i)},\,\,  \widetilde{T}_2^{\,\,(i)},\, \ldots,\, \widetilde{T}_i^{\,\,(i)}$$
satisfying the following properties
\begin{eqnarray}\label{eq:Ttil-ij}
\widetilde{T}_i^{\,\,(i)} &=& T_i = t_{2i+2}-t_{2i} \\
\widetilde{T}_j^{\,\,(i)} - \delta \,\widetilde{T}_{j-1}^{\,\,(i)} &=& -\tau_1 \, \log d-\tau_2 \, \log b +(1-\delta) (s_1+s_2)  \quad \quad \forall \, j \in \{1, 2, \ldots, i\}. \nonumber
\end{eqnarray}
\medskip
By finite induction, it is straightforward that, for every $i \, \in \, \NN$,

\begin{equation}\label{eq:T0i}
\widetilde{T}_0^{\,\,(i)}=\frac{T_i + (\sum_{j=0}^{i-1} \,\delta^j) \,  -\tau_1 \, \log d-\tau_2 \, \log b +(1-\delta) (s_1+s_2)}{\delta^i}.
\end{equation}
Therefore, using the argument of \cite{CR2017}, we may conclude that:

\begin{lemma}\label{le:calculus2} Let $P_0=(\rho_0, \theta_0,1)$ be a point in $\Out^+(\C_2)$ and take the corresponding sequence $(t_j)_{j\,\in\, \NN_0}$. Then, for each $i \in \NN$, there exists $J_i \in \mathbb{R}$ such that $\,\sum_{i=1}^\infty i\,|J_i|<\infty$ and
$$(t_{2i+2}-t_{2i})- \delta\,(t_{2i}-t_{2i-2})= -\tau_1\, \log b -\tau_2 \,\log d  +  J_i.$$
\medskip
In addition, for every $i \, \in \, \NN_0$, we have $\,\widetilde{T}_0^{\,\,(i+1)} - \,\widetilde{T}_0^{\,\,(i)} = \frac{J_{i+1}}{\delta^{i+1}}.$
\end{lemma}

\medskip

As $\delta>1$, the series $\sum_{j=1}^\infty \frac{J_j}{\delta^{j}} $ converges, and so the sequence $\left(\widetilde{T}_0^{\,\,(i)}\right)_{i \,\in\, \NN_0}$ converges. Denote its limit by $\widetilde{T}_0$:
\begin{equation}\label{T0}
\widetilde{T}_0:= \lim_{i \to + \infty}\, \widetilde{T}_0^{\,\,(i)} = T_0^{(0)} + \sum_{j=1}^\infty \frac{J_j}{\delta^{j}} = T_0 +   \sum_{j=1}^\infty \frac{J_j}{\delta^{j}}.
\end{equation}

Next, for $i \geq 1$, consider the sequence $(\widetilde{T}_i)_{i \,\in\, \NN_0}$ satisfying
\begin{equation}\label{eq:Ttil}
\widetilde{T}_i = \delta \, \widetilde{T}_{i-1} -\tau_1 \, \log d-\tau_2 \, \log b +(1-\delta) (s_1+s_2) \qquad \forall \, i \, \in \, \NN
\end{equation}
where $\widetilde{T}_0$ was computed in \eqref{T0}.

\begin{lemma}\cite{CR2017}
\label{le:convergence} The series $\sum_{i=0}^{+\infty} \, (T_i - \widetilde{T}_{i})$ converges and $\lim_{i \to +\infty} (T_i - \widetilde{T}_i) = 0.$
\end{lemma}
Therefore, we may take a sequence $\left(\widetilde{t}_{2i}\right)_{i \, \in \, \NN_0}$ of positive real numbers such that
\begin{eqnarray}\label{eq:times-even-til}
\widetilde{t}_{0} &=& 0  \nonumber\\
\widetilde{T}_i &=& \widetilde{t}_{2i+2}- \widetilde{t}_{2i} \nonumber\\
\lim_{i\, \to \, + \infty}\, (t_{2i} - \widetilde{t}_{2i}) &=& 0.
\end{eqnarray}
Moreover, by construction (see \eqref{eq:Ttil}) we have
\begin{equation}\label{eq:tau log a}
(\widetilde{t}_{2i+2}- \widetilde{t}_{2i}) - \delta\, (\widetilde{t}_{2i}- \widetilde{t}_{2i-2}) = -\tau_1 \, \log d-\tau_2 \, \log b +(1-\delta) (s_1+s_2).
\end{equation}
After defining the sequences of even indices, we take a sequence $\left(\widetilde{t}_{2i+1}\right)_{i \, \in \, \NN_0}$ satisfying, for every $i \, \in \, \NN_0$,
\begin{equation}\label{eq:times-odd-til}
\widetilde{t}_{2i+2}- \widetilde{t}_{2i+1} = \gamma_1 \, (\widetilde{t}_{2i+1} - \widetilde{t}_{2i}) - \frac{1}{E_2} \log b + (s_2 - \gamma_1 s_1).
\end{equation}

\begin{lemma}[\cite{CR2017}]\label{le:prop-t-til} $\,$
\begin{enumerate}
\item $\lim_{i\, \to \, + \infty}\, (t_{2i+1} - \widetilde{t}_{2i+1}) = 0.$ \\
\item $\lim_{i \to +\infty} \,(\widetilde{t}_{2i+1} - \widetilde{t}_{2i}) - \gamma_2\,(\widetilde{t}_{2i} - \widetilde{t}_{2i-1}) = -\,\frac{1}{E_1}\log d + (s_1 - \gamma_2 \,s_2).$\\
\item $\lim_{i \to +\infty} \,(\widetilde{t}_{2i+2} - \widetilde{t}_{2i+1}) - \gamma_1\,(\widetilde{t}_{2i+1} - \widetilde{t}_{2i}) = -\,\frac{1}{E_2}\log b + (s_2 - \gamma_1 \, s_1).$\\
\end{enumerate}
\end{lemma}

As any solution of $f$ in $\mathfrak{B}(\mathcal{H}_f)$ eventually hits $\Out\,(\C_2)$, we may apply the previous construction to all the orbits of $f$ in $\mathfrak{B}(\mathcal{H}_f)$. So, given any $P_0 \in \mathfrak{B}(\mathcal{H}_f)$, we take the first non-negative hitting time of the forward orbit of $P_0$ at $\Out\,(\C_2)$, defined by
$$t_{\Sigma_2} (P_0) = \min \,\{t \in \RR^+_0 \colon \varphi(t,\,P_0) \in \Out\,(\C_2)\}.$$
As $\Out^+(\C_2)$ and $\Out^-(\C_2)$ are relative-open sets, this first-hitting-time map is continuous with $P_0$. Then, having fixed
$$P = \varphi(t_{\Sigma_2}(P_0),\,P_0) = \left(\rho_0,\,\theta_0, \pm 1\right) \in \Out\,(\C_2)$$
we consider its hitting times sequence $\left(t^{\,\,(P)}_{i}\right)_{i \, \in \, \NN_0}$
and build the sequence $\left(\widetilde{t}^{\,\,\,(P)}_{i}\right)_{i \, \in \, \NN_0}$ as explained in the previous section.

Adjusting the cross sections $\Sigma_1$ and $\Sigma_2$ if needed, we now find a point $\widetilde{P} \in \Out\,(\C_2)$ in the $f-$trajectory of $P$ whose hitting times sequence is precisely $\left(\widetilde{t}^{\,\,\,(P)}_{i}\right)_{i \, \in \, \NN_0}$. Notice that the new cross sections are close to the previous ones since the sequences $\left(t_{i}\right)_{i \, \in \, \NN_0}$ and $\left(\widetilde{t}_{i}\right)_{i \, \in \, \NN_0}$ are uniformly close. We are left to show that there exists a continuous choice of such a trajectory with hitting times sequence $\left(\widetilde{t}^{\,\,\,(P)}_{i}\right)_{i \, \in \, \NN_0}$.

\subsubsection{\textbf{\emph{Coordinates of $\widetilde{P}$}}}\label{sse:special-point}

Given a sequence of times $\left(\widetilde{t}_i\right)_{i \, \in \, \NN_0}$ satisfying  $\widetilde{t}_0=0$ and the properties established in Lemma~\ref{le:prop-t-til}, \eqref{eq:times-even-til}, \eqref{eq:tau log a} and \eqref{eq:times-odd-til}, one may recover from its terms the coordinates of a point $(\rho_0, \,\theta_0, 1) \in \Out^+({\C}_2)$ whose $i$th hitting time is precisely $\widetilde{t}_i$. Firstly, we solve the equation (see \eqref{eq:t1})
\begin{equation}\label{eq:z}
\widetilde{t}_1 = -\,\frac{1}{{E}_1} \log \,({d}\, |\rho_0 - {R}_2|)+{s}_1
\end{equation}
obtaining $\rho_0$. Then, using \eqref{eq:t2}, we get
\begin{equation}\label{eq:rho}
\widetilde{t}_2 = \widetilde{t}_1 +{s}_2 -\frac{1}{{E}_2} \log   \,({b}\, |\rho_1 - {R}_1|)
\end{equation}
and compute $\rho_1$. And so on, getting from such a sequence of times all the values of the radial coordinates $\left(\rho_{2i+1}\right)_{i \, \in \, \NN_0}$ and $\left(\rho_{2i}\right)_{i \, \in \, \NN_0}$ of the successive hitting points at $\Out^+({\C}_1)$ and $\Out^+({\C}_2)$, respectively.

\medskip

Notice that the previous computations do not depend on the angular coordinate. That is why nothing has yet been disclosed about $\theta_0$ from them. Concerning the evolution in $\mathbb{R}^+$ of the angular coordinates, the spinning in average inside the cylinders is given, for every $i \in \mathbb{N}_0$, by
\begin{eqnarray}\label{eq:spin}
\frac{\theta_{2i+2} - c\,\theta_{2i}}{\widetilde{t}_{2i+2}-\widetilde{t}_{2i}} &=& \frac{(\theta_{2i+2} - a\,\theta_{2i+1}) + (a\,\theta_{2i+1} - c\,\theta_{2i})}{\widetilde{t}_{2i+2}-\widetilde{t}_{2i}} \nonumber\\
&=& \frac{\omega_2\,(\widetilde{t}_{2i+2} - \widetilde{t}_{2i+1}) + \omega_1\,(\widetilde{t}_{2i+1} - \widetilde{t}_{2i})}{\widetilde{t}_{2i+2}-\widetilde{t}_{2i}} \nonumber\\
&=& \frac{\omega_1 + \gamma_1\, \omega_2}{\gamma_1 + 1}
\end{eqnarray}
(cf. Corollary~\ref{cor:speed-of-convergence}). Moreover, Lemma~\ref{le:prop-t-til} indicates that
$$\frac{\theta_{2i+1} - c\,\theta_{2i}}{\theta_{2i+2}- a\,\theta_{2i+1}} = \frac{\omega_1\,(\widetilde{t}_{2i+1} - \widetilde{t}_{2i})}{\omega_2\,\left(\widetilde{t}_{2i+2} - \widetilde{t}_{2i+1}\right)} = \frac{\omega_1}{\gamma_1 \,\omega_2}.$$
So
\begin{eqnarray*}
\theta_{2i+2} - \theta_{2i} &=& (\theta_{2i+2} - a\,\theta_{2i+1}) + \left(a\, \theta_{2i+1} - a\, c\,\theta_{2i}\right) + \left(a\,c - 1\right)\,\theta_{2i}\\
&=& (\theta_{2i+2} - a\,\theta_{2i+1})\,\left(\frac{a\,\omega_1}{\gamma_1 \,\omega_2} + 1\right) + \left(a\,c - 1\right)\,\theta_{2i}\\
&=& \omega_2\,(\widetilde{t}_{2i+2} - \widetilde{t}_{2i+1})\,\left(\frac{a\,\omega_1}{\gamma_1 \,\omega_2} + 1\right) + \left(a\,c - 1\right)\,\theta_{2i} \\
&=& \frac{a\,\omega_1 + \gamma_1\, \omega_2}{\gamma_1}\,\,(\widetilde{t}_{2i+2} - \widetilde{t}_{2i+1}) + \left(a\,c - 1\right)\,\theta_{2i}.
\end{eqnarray*}
On the other hand, from \eqref{eq:spin} we get
\begin{eqnarray*}
\theta_{2i+2} - \theta_{2i} &=& \left(\theta_{2i+2} - c\,\theta_{2i}\right) + \left(c - 1\right)\,\theta_{2i} \\
&=& \frac{\omega_1 + \gamma_1\, \omega_2}{\gamma_1 + 1}\,\,(\widetilde{t}_{2i+2}-\widetilde{t}_{2i}) + \left(c - 1\right)\,\theta_{2i}.
\end{eqnarray*}
Consequently,
\begin{equation*}
\frac{a\,\omega_1 + \gamma_1\, \omega_2}{\gamma_1}\,(\widetilde{t}_{2i+2} - \widetilde{t}_{2i+1}) + \left(a\,c - 1\right)\,\theta_{2i} = \frac{\omega_1 + \gamma_1\, \omega_2}{\gamma_1 + 1}\,(\widetilde{t}_{2i+2}-\widetilde{t}_{2i}) + \left(c - 1\right)\,\theta_{2i}
\end{equation*}
or, equivalently,
\begin{eqnarray}\label{eq:theta}
\theta_{2i}\Big(c\,(a-1)\Big) &=& \frac{\omega_1 + \gamma_1\, \omega_2}{\gamma_1 + 1} \,\left(\widetilde{t}_{2i+2} - \widetilde{t}_{2i+1}\right) - \frac{a\,\omega_1 + \gamma_1\, \omega_2}{\gamma_1}\,\left(\widetilde{t}_{2i+2} - \widetilde{t}_{2i}\right).
\end{eqnarray}

Similar estimates show that

\begin{eqnarray}
\frac{\theta_{2i+1} - a\,\theta_{2i-1}}{\widetilde{t}_{2i+1}-\widetilde{t}_{2i-1}} &=& \frac{\omega_2 + \gamma_2\, \omega_1}{\gamma_2 + 1} \nonumber \\ \nonumber\\
\theta_{2i+1} - \theta_{2i-1} &=& \frac{\omega_2 + \gamma_2\, \omega_1}{\gamma_2 + 1}\,\,(\widetilde{t}_{2i+1}-\widetilde{t}_{2i-1}) + \left(a - 1\right)\,\theta_{2i-1} \nonumber \\ \nonumber\\
\theta_{2i+1} - \theta_{2i-1} &=& \frac{c\,\omega_2 + \gamma_2\, \omega_1}{\gamma_2}\,\,(\widetilde{t}_{2i+1} - \widetilde{t}_{2i}) + \left(a\,c - 1\right)\,\theta_{2i-1} \nonumber \\ \nonumber\\
\theta_{2i-1}\Big(a\,(c-1)\Big) &=& \frac{\omega_2 + \gamma_2\, \omega_1}{\gamma_2 + 1} \,\left(\widetilde{t}_{2i+1} - \widetilde{t}_{2i-1}\right) - \frac{c\,\omega_2 + \gamma_2\, \omega_1}{\gamma_2}\,\left(\widetilde{t}_{2i+1} - \widetilde{t}_{2i}\right).
\end{eqnarray}
\medskip

From these computations the angular coordinate $\theta_0$ is uniquely determined if and only if either $a \neq 1$, in which case
$$\theta_{0} = \Big(\frac{1}{c\,(a-1)}\Big) \, \Big[\frac{\omega_1 + \gamma_1\, \omega_2}{\gamma_1 + 1} \,\left(\widetilde{t}_{2} - \widetilde{t}_{1}\right) - \frac{a\,\omega_1 + \gamma_1\, \omega_2}{\gamma_1}\,\left(\widetilde{t}_{2} - \widetilde{t}_{0}\right)\Big]$$
or $c \neq 1$, in which case
$$\theta_{1} = \Big(\frac{1}{a\,(c-1)}\Big) \, \Big[\frac{\omega_2 + \gamma_2\, \omega_1}{\gamma_2 + 1} \,\left(\widetilde{t}_{3} - \widetilde{t}_{1}\right) - \frac{c\,\omega_2 + \gamma_2\, \omega_1}{\gamma_2}\,\left(\widetilde{t}_{3} - \widetilde{t}_{2}\right)\Big]$$
is known, from which $\theta_0$ is found iterating the flow backwards.

If $a = 1 = c$, we may evaluate $\theta_2 - \theta_0$, but all possible values $\theta_0 \in [0, 2\pi[$ are good choices for the angular coordinate. In particular, in this case, the invariants $\frac{\omega_1 + \gamma_1\, \omega_2}{1 + \gamma_1}$ and $\frac{\omega_2 + \gamma_2\, \omega_1}{1 + \gamma_2}$ are not used to construct the conjugacy.

\subsection{The conjugacy}
Consider linearizing neighborhoods of $\overline{\C_1}$ and $\overline{\C_2}$, the periodic solutions of $g$, and take a point $\overline{P}=(\rho_0,\theta_0,1) \in {\Out^+}(\overline{\C_2})$, the corresponding hitting times sequence $\left(t_i\right)_{i \, \in \, \NN_0}$ at cross sections ${\Inn^+(\overline{\C_1}})$ and $\Sigma_2$, and the sequence of times $\left(\widetilde{t}_i\right)_{i \, \in \, \NN_0}$ obtained in Subsection~\ref{sse:time2}.

As done for $f$ in Subsection~\ref{sse:special-point}, using estimates similar to \eqref{eq:z}, \eqref{eq:rho} and \eqref{eq:theta}, we now find for $g$ a unique point $Q_P$, given in local coordinates by $(\overline{\rho_0},\overline{\theta_0},1)$, where
\begin{eqnarray*}\label{eq:point-for-g}
\overline{\rho_0}&=& R_2 \,\pm \,\frac{e^{-(\tilde{t}_1  - \overline{s}_1)\,\overline{E}_1}}{d} \, \\
\overline{\theta_{0}} &=& \Big(\frac{1}{c\,(a-1)}\Big) \, \Big[\frac{\omega_1 + \gamma_1\, \omega_2}{\gamma_1 + 1} \,\left(\widetilde{t}_{2} - \widetilde{t}_{1}\right) - \frac{a\,\omega_1 + \gamma_1\, \omega_2}{\gamma_1}\,\left(\widetilde{t}_{2} - \widetilde{t}_{0}\right)\Big] \quad \quad \text{if $a \neq 1$}\\
\overline{\theta_{1}} &=& \Big(\frac{1}{a\,(c-1)}\Big) \, \Big[\frac{\omega_2 + \gamma_2\, \omega_1}{\gamma_2 + 1} \,\left(\widetilde{t}_{3} - \widetilde{t}_{1}\right) - \frac{c\,\omega_2 + \gamma_2\, \omega_1}{\gamma_2}\,\left(\widetilde{t}_{3} - \widetilde{t}_{2}\right)\Big] \quad \quad \text{if $c \neq 1$} \\
\overline{\theta_{0}} &=& \text{any value in $[0, 2\pi[$} \quad \quad \text{if $a = 1 = c$}.
\end{eqnarray*}
\medskip
The set of these points build cross sections $\overline{\Sigma_1}$ and $\overline{\Sigma_2}$ for $g$ at which the points $Q_P$ have the prescribed hitting times $\left(\widetilde{t}_i\right)_{i \, \in \, \NN_0}$ by the action of $g$. Next, we take the map
$$H \colon \,\,P \in \Sigma_2 \cap \Out^+(\C_2) \quad \mapsto \quad Q_P$$
and extend it using the flows $\varphi$ and $\overline{\varphi}$ of $f$ and $g$, respectively: for every $t \in \RR$, set $H(\varphi_t(P))=\overline{\varphi}_t(H(P)).$ An analogous construction is repeated for $\Out^-(\C_2)$.

\begin{lemma} [\cite{CR2017}]
$H$ is a conjugacy.
\end{lemma}

This ends the proof of Theorem~\ref{teo:maintheorem-1}.\\

\section{Final remark}\label{final}

The proof of Theorem \ref{teo:maintheorem-1} may be easily adapted to the case $\C_1 = \C_2$, thereby providing a complete set of invariants for an attracting homoclinic cycle associated to a periodic solution of a vector field in $\mathfrak{X}^r_{\text{PtoP}}(\EU^3)$, subject to the condition \eqref{times1}. More precisely, the corresponding complete set of invariants reduces to
$$\left\{\wp_1, \, \,\gamma_1, \, \,\omega_1, \,   \,-\frac{1}{E_1} \log b + s_1\, (1-\gamma_1)\right\}.$$

Regarding the construction of invariants under conjugacy for homoclinic cycles of a vector field, we refer the reader to \cite{Togawa}, where Togawa analyzes a homoclinic cycle of a saddle-focus and shows, using a knot-like argument, that the saddle-index is a conjugacy invariant;  to the paper  \cite{Arnold}, where Arnold \emph{et al} prove that the saddle-index is in fact an invariant under topological equivalence; and to the work  \cite{Dufraine} whose author, in the same setting, describes a new invariant under conjugacy given by the absolute value of the imaginary part of the complex eigenvalues of the saddle-focus. The search for a complete set of invariants for {more general} homoclinic cycles associated to either a saddle-focus or a periodic solution is still an open problem.

\section{An example}\label{Example}

In this section we present a family of vector fields in $\RR^3$ satisfying properties (P1)--(P6) obtained from Bowen's example presented in \cite{Takens94}. The latter is a $C^\infty$ vector field in the plane with structurally unstable connections between two equilibria. We will use the technique introduced in \cite{Chossat86} and further explored in \cite{Melbourne, ACL06}, combined with symmetry breaking, to lift Bowen's example to a vector field in $\RR^3$ with periodic solutions involved in a heteroclinic cycle satisfying the conditions stated in Section~\ref{se:hypotheses}.

\subsection{Lifting and its properties}\label{lifting}

The authors of \cite{ACL06, Rodrigues} investigate how some properties of a $\ZZ_2$--equivariant vector field on $\RR^n$ lift by a rotation to properties of a corresponding vector field on $\RR^{n+1}$. For the sake of completeness,  we review some of these properties. Let $X_n$ be a $\ZZ_2$--equivariant vector field on $\RR^n$. Without loss of generality, we may assume that $X_n$ is equivariant by the action of
$$T_n(x_1, x_2, ...., x_{n-1}, y)=  (x_1, x_2, ...., x_{n-1}, - y).$$
The vector field $X_{n+1}$ on $\RR^{n+1}$ is obtained by adding the auxiliary equation $\dot\theta=\omega>0$ and interpreting $(y, \theta)$ as polar coordinates. In cartesian coordinates $(x_1, . . . , x_{n-1}, r_1, r_2) \in \RR^{n+1}$, this extra equation corresponds to the system $r_1 = |y| \cos \theta$ and $r_2 = |y| \sin \theta$. The resulting vector field $X_{n+1}$ on $\RR^{n+1}$ is called the \emph{lift by rotation of} $X_n$, and is $\mathbb{SO}(2)$--equivariant in the last two coordinates.
\medbreak

Given a set $\Lambda \subset \RR^n$, let $\mathcal{L}(\Lambda)\subset \RR^{n+1}$ be the lift by rotation of $\Lambda$, that is,
$$\Big\{(x_1,...,x_{n-1}, r_1, r_2) \in \RR^{n+1} \colon \, (x_1,\ldots,x_{n-1},||(r_1,r_2)||) \quad \text{or} \quad (x_1,\ldots,x_{n-1},-||(r_1,r_2)||) \in \Lambda\Big\}.$$
It was shown in \cite[Section 3]{ACL06} that, if $X_n$ is a $\ZZ_2(T_n)$--equivariant vector field in $\RR^n$ and $X_{n+1}$ is its lift by rotation to $\RR^{n+1}$, then:
\medskip
\begin{enumerate}
\item If $p$ is a hyperbolic equilibrium of $X_n$, then $\mathcal{L}(\{p\})$ is a hyperbolic periodic orbit of $X_{n+1}$ with minimal period $\frac{2\pi}{\omega}$.
\medskip
\item If $[p_1 \to p_2]$ is a $k$-dimensional heteroclinic connection between equilibria $p_1$ and $p_2$ and it is not contained in $\text{Fix}(\ZZ_2(T_n))$, then it lifts to a $(k + 1)$-dimensional connection between the periodic orbits $\mathcal{L}(\{p_1\})$ and $\mathcal{L}(\{p_2\})$ of $X_{n+1}$.
\medskip
\item If $\Lambda$ is a compact $X_n$--invariant asymptotically stable set, then $\mathcal{L}(\Lambda)$ is a compact $X_{n+1}$--invariant
asymptotically stable set.
\end{enumerate}

\subsection{Bowen's example}\label{subsecBowen} Consider the system of differential equations
\begin{equation}
\label{example 1}
\left\{
\begin{array}{l}
\dot{x}=-y \\
\dot{y}=x-x^3
\end{array}
\right.
\end{equation}
whose equilibria are $\mathcal{O} = (0,0)$ and $P^\pm=(\pm 1, 0)$. This is a conservative system, with first integral given by
$$\vv(x,y) = \frac{x^2 }{2}\left(1-\frac{x^2}{2} \right)+\frac{y^2}{2}.$$
It is easy to check that the origin $\mathcal{O}$ is a center. The equilibria $P^\pm$ are saddles  with eigenvalues $\pm \sqrt{2}$. They are contained in the $\vv$-energy level $\vv \equiv 1/4$, and therefore there are two one-dimensional connections between them, one from $P^+$ to $P^-$ and another from $P^-$ to $P^+$, we denote by $[P^+ \rightarrow P^-]$ and $[P^- \rightarrow P^+]$, respectively. Let $\mathcal{H}_0$ be this heteroclinic cycle. The open domain $\mathcal{D}$ bounded by $\mathcal{H}_0$ and containing $\mathcal{O}$ is filled by closed trajectories and we have $0 \le \vv < 1/4$. Notice also that the boundary of $\mathcal{D}$ intersects the line $x=0$ at the points $(0,\pm\sqrt{2}/2)$. See Figure~\ref{example1}.

\begin{figure}[h]
\begin{center}
\includegraphics[height=8cm]{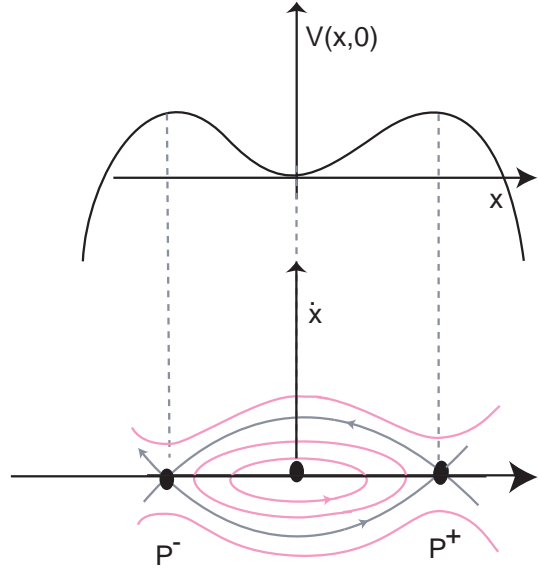}
\end{center}
\caption{\small Phase diagram of \eqref{example 1}.}
\label{example1}
\end{figure}

\subsection{A perturbation of Bowen's example}
Given $\varepsilon >0$, consider the following perturbation of \eqref{example 1} defined by the differential equations
\begin{equation}
\label{example 2}
\left\{
\begin{array}{l}
\dot{x}=-y \\
\dot{y}=x-x^3 -\varepsilon \,y\left(\vv(x,y)-\frac{1}{4}\right).
\end{array}
\right.
\end{equation}
For $\varepsilon > 0$ small enough, the heteroclinic cycle $\mathcal{H}_0$ persists, but now the $\omega$-limit of every trajectory with initial condition in $\mathcal{D}\setminus\{(0,0)\}$ is $\mathcal{H}_0$. Check these details in Figure~\ref{example2}.

\begin{figure}[h]
\begin{center}
\includegraphics[height=3cm]{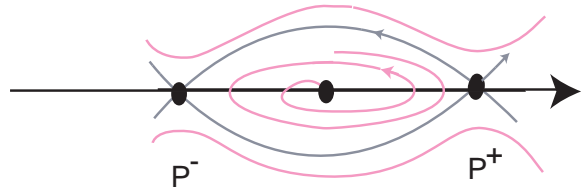}
\end{center}
\caption{\small  Bowen's example \eqref{example 2} with $\varepsilon>0$. }
\label{example2}
\end{figure}

\subsection{The lifting of Bowen's cycle}
According to the lifting procedure described above, we now construct a vector field on $\RR^3$ with two periodic solutions linked in a cyclic way within a configuration similar to the heteroclinic cycle $\mathcal{H}_0$ of Bowen's example. Noticing that $\mathcal{H}_0$ is contained in the half plane $y > -1$, one rotates the phase diagram of Bowen's perturbed example around the line $y=-1$. This transforms the equilibria $P^\pm$ into saddle periodic solutions as in (P1), and the one-dimensional heteroclinic connections into two-dimensional ones which are diffeomorphic to cylinders as in (P2). Meanwhile, the attracting character of the cycle $\mathcal{H}_0$ is preserved and one connected component of the stable manifold of each periodic solution coincides with a connected component of the unstable manifold of the other as demanded in (P3).

More precisely, in the region $y > - 1$, we may write $y+1 = r^2$ for a unique $r > 0$, and with $\dot{r}= \frac{\dot{y}}{2r}$ the system of equations \eqref{example 2} takes the form
$$\left\{
\begin{array}{l}
\dot{x}= 1-r^2\\
\dot{r}= \frac{1}{2r}\left[x-x^3-\varepsilon\left(\frac{x^2}{2}- \frac{x^4}{4} + \frac{(r^2-1)^2}{2}-\frac{1}{4}\right)(r^2-1)\right].
\end{array}
\right.
$$
Multiplying both equations by the positive term $2r^2$ does not qualitatively affect the phase portrait, thus \eqref{example 2} in the region $y>-1$ is equivalent to
\begin{equation}\label{example 3}
\left\{
\begin{array}{l}
\dot{x}= 2r^2(1-r^2) \\
\dot{r}= r\left(x-x^3-\varepsilon \left(\frac{x^2}{2}- \frac{x^4}{4} + \frac{(r^2-1)^2}{2}-\frac{1}{4}\right)(r^2-1)\right)
\end{array}
\right.
\end{equation}
in the domain $r>0$. It is straightforward to check that the system of equations \eqref{example 3} for $(x,r) \in \RR^2$ has the following properties:
\begin{enumerate}
\item The line $r=0$ is flow-invariant.
\medskip
\item It is $\ZZ_2(\Gamma)$--equivariant, where $\Gamma(x, r)=(x, -r)$.
\end{enumerate}
\medskip
This allows us to apply the lifting procedure as described above, performing the mentioned rotation of the phase diagram of \eqref{example 3}: adding a new variable $\theta$ with $\dot{\theta}=\omega$, for some constant $\omega >0$ and taking Cartesian coordinates  $(x,r_1,r_2) = (x, r\cos \theta, r\sin \theta)$, the system of equations \eqref{example 3} becomes
\begin{equation}\label{example 4}
\left\{\begin{array}{l}
\dot{x}= 2(1 - r_1^2 - r_2^2)(r_1^2 + r_2^2) \\ \\
\dot{r}_1 = r_1 \left[x-x^3-\varepsilon(r_1^2 + r_2^2 - 1)\left(\frac{x^2}{2} - \frac{x^4}{4} +
 \frac{(r_1^2 + r_2^2 - 1)}{2}-\frac{1}{4}\right)\right] -\omega r_2\\ \\
\dot{r}_2 = r_2 \left[x-x^3-\varepsilon(r_1^2 + r_2^2 - 1)\left(\frac{x^2}{2} - \frac{x^4}{4} +
\frac{(r_1^2 + r_2^2 - 1)}{2} - \frac{1}{4}\right)\right] +\omega r_1.
\end{array}
\right.
\end{equation}
The equilibria $P^+$ and $P^-$ lift to two hyperbolic closed orbits satisfying (P1), namely
\begin{eqnarray*}
\mathcal{C}_1 &:=& \Big\{(x,r_1,r_2) \colon \, x=1, \quad  r_1^2 + r_2^2=1\Big\} \\
\mathcal{C}_2 &:=& \Big\{(x,r_1,r_2) \colon \, x=-1, \quad  r_1^2 + r_2^2=1\Big\}
\end{eqnarray*}
with radius $R_1 = R_2 = 1$. The Floquet multipliers of $\C_1$ and $\C_2$ are given by $e^{\sqrt{2}}>1$ and $e^{-\sqrt{2}}<1$ (details in \cite{Field}). Their two-dimensional stable and unstable manifolds are homeomorphic to cylinders and, for $\varepsilon>0$ small enough, the flow of \eqref{example 4} has a heteroclinic cycle $\mathcal{H}$ as stated in (P2) and (P3).

Admittedly, conditions $C_1 > E_1$ and $C_2 > E_2$ of item (P1) fail, and so Krupa-Melbourne's criterium of \cite{KM1, KM2} is no longer applicable. However, by construction, $\mathcal{H}_0$ is asymptotically stable, and so is $\mathcal{H}$. As explained in Subsection \ref{lifting}, the basin of attraction of $\mathcal{H}$ contains $\mathcal{L}(\mathcal{D}\backslash\{(0,0)\})$. In what follows, $f_\mathcal{B}$ stands for the vector field just obtained as the lifting of the perturbed version of Bowen's example.

\subsection{Checking conditions (P4) and (P5) for $f_\mathcal{B}$} $\,$
For the unlifted system \eqref{example 2}, we may choose $\varepsilon>0$ and $K>0$ to define global sections
\begin{eqnarray*}
\Out (P^+) &=& \Big\{(x,y): x=1-\varepsilon, \quad  y\in [0, K \varepsilon] \Big\} \\
\Inn (P^-) &=& \Big\{(x,y): x=-1+\varepsilon,\quad  y\in [0, K \varepsilon] \Big\}
\end{eqnarray*}
and, in a similar way, the sections $\Out (P^-)$ and $\Inn (P^+)$. Therefore, the cross sections for  \eqref{example 2} may be written as
\begin{eqnarray*}
\Out (\C_1) &=& \Big\{(x,r_1, r_2) \colon x = 1-\varepsilon, \quad r_1^2 + r_2^2 \,\in \,[1,\,1+ K \varepsilon] \Big\} \\
\Inn(\C_2) &=& \Big\{(x,r_1, r_2) \colon x = -1+\varepsilon, \quad r_1^2+r_2^2 \,\in \,[1,\,1+ K \varepsilon] \Big\}
\end{eqnarray*}
and similarly for $\Out (\C_2)$ and $\Inn (\C_1)$. If $r_1r_2\neq 0$, changing coordinates as follows
$$\rho \leftrightarrow \sqrt{r_1^2+r_2^2} \quad \quad \quad \theta  \leftrightarrow \text{arctan} \left(\frac{r_2}{r_1}\right) + m\pi, \quad m=0,1 \quad \quad \quad z \leftrightarrow x$$
we identify $(x, r_1, r_2)$ with $(\rho, \theta, z)$ as done in Section~\ref{se:Local}. Hence the transition from $\Out (\C_1)$ to $\Inn(\C_2)$ maps $(\rho_0, \theta_0, \varepsilon)$ to $(R_1+\varepsilon, \theta_1, z_1)= (1+\varepsilon, \theta_1, z_1)$ and is linear, with a diagonal matrix given in the cylindrical coordinates $(\rho, \theta, z)$ by the matrix
$\tiny
\left[ {\begin{array}{ccc}
1 & 0 & 0 \\
0 & a & 0 \\
0 & 0 & b \\  \end{array} }
\right]$
for some $a>0$ and $b>0$. The same argument applies to the connection $[\C_2 \rightarrow \C_1]$. This completes the verification of condition (P4).\\

In order to characterize the first return map to the cross sections of lifted system \eqref{example 4}, we add the following assumptions to the vector field \eqref{example 3}:
\medskip

\textbf{(H1):} There are $s_1 \geq 0$ and an open set $U_1 \subset \Out(P^+)$ containing $W^u(P^+)$ such that the transition time to $\Inn (P^-)$ of all trajectories starting in $U_1$ is constant and equal to $s_1$. The transition from $U_1$ to $\Inn(P^-)$ maps $(1-\varepsilon, \,y)$ to $(-1+\varepsilon, \,b\,y)$. \\

\textbf{(H2):} Analogously, there are $s_2 \geq 0$ and an open set $U_2 \subset \Out(P^-)$ containing  $W^u(P^-)$ such that the transition time to $\Inn (P^+)$ of all trajectories starting in $U_2$ is constant and equal to $s_2$. The transition from $U_2$ to $\Inn(P^+)$ maps $(-1+\varepsilon, \,y)$ into $(1-\varepsilon, \,d\,y)$.\\

By construction, property (P6) is guaranteed. We now proceed to check condition (P5).
\begin{lemma} $\,$
\begin{enumerate}
\item For $j \in \{1,2\}$, the transition times are constant on $\mathcal{L}(U_j)$ and equal to $s_j$.
\item The angular speeds of the periodic solutions $\C_1$ and $\C_2$ are equal to $\omega$.
\end{enumerate}
\end{lemma}

\begin{proof}
Item (1) follows from the way the lifting is carried out, ensuring that the global cross sections $\Inn(\C_1)$, $\Inn(\C_2)$, $\Out(\C_1)$ and $\Out(\C_2)$ are lifts by rotation of $\Inn(P^+)$, $\Inn(P^-)$, $\Out(P^+)$ and $\Out(P^-)$, respectively. Using \textbf{(H1)}, if $P \in \mathcal{L}(U_1)\subset \Out (\C_1)$, then the transition time of its trajectory to $\Inn (C_2)$ is $s_1$. Analogous conclusion holds for $P \in \mathcal{L}(U_2)$ using \textbf{(H2)}. Part (2) of the statement is a consequence of the fact that the solutions corresponding to the periodic solutions are parameterized by $t \mapsto (\pm 1, \ \cos(\omega t),\ \sin(\omega t))$.
\end{proof}

Figure ~\ref{lifting1} summarizes the previous information concerning the lifted dynamics.

\begin{figure}[h]
\begin{center}
\includegraphics[height=13.4cm]{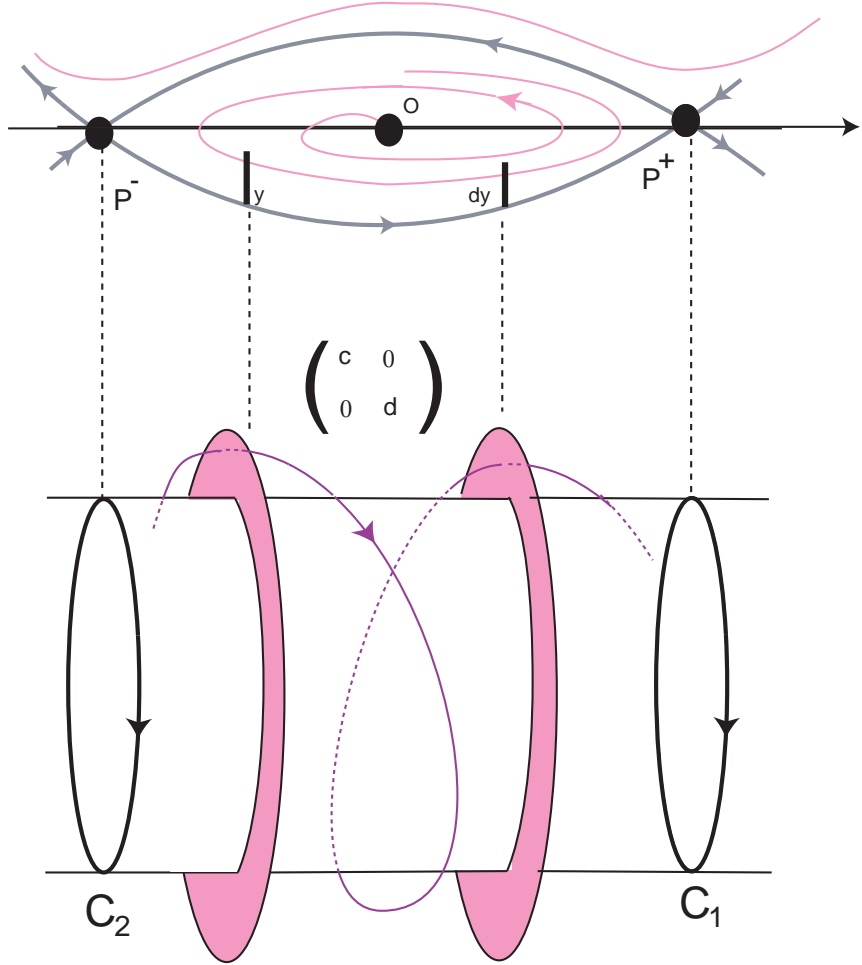}
\end{center}
\caption{\small  Illustration of the properties that are conveyed from \eqref{example 3} to its lifting \eqref{example 4}.}
\label{lifting1}
\end{figure}

\subsection{Invariants for $f_\mathcal{B}$}

Now Theorem~\ref{teo:maintheorem-1} applies to the heteroclinic cycle $\mathcal{H}$ and its basin of attraction (which contains ${\mathcal{L}(\mathcal{D} \backslash\{(0,0)\})}$) of the example \eqref{example 4}, indicating that the set
$$\left\{\omega, \, \gamma_1, \,\gamma_2, \, \,-\frac{1}{E_1} \log d + (s_1-\gamma_1\,s_2), \,-\frac{1}{E_2} \log b + (s_2-\gamma_2 \,s_1)\right\}$$
is a complete family of invariants for $f_\mathcal{B}$ under topological conjugacy in ${\mathcal{L}(\mathcal{D}\backslash\{(0,0)\})}$. In addition, for the example \eqref{example 4} we have $E_1 = E_2 = \sqrt{2}$ and $\gamma_1 = \gamma_2 = 1$. The values of the constants $s_1$ and $s_2$ depend on the chosen cross sections for the perturbed Bowen's example.

\end{document}